\documentclass{article}
\usepackage[shortlabels]{enumitem}
\usepackage{nameref}
\usepackage{hyperref}
\usepackage[nospace,noadjust]{cite}
\usepackage{setspace}
\usepackage{graphicx}
\usepackage[all]{xy}
\usepackage{amsmath, amsthm, amssymb, amsfonts, tensor}
\usepackage{verbatim}
\usepackage{url}
\usepackage[affil-it]{authblk}

\makeatletter
\newcommand{\displaybump}{\hbox to \@totalleftmargin{\hfil}}
\makeatother

\SelectTips{cm}{}

\DeclareMathOperator\stab{stab}
\DeclareMathOperator\Sym{Sym}
\DeclareMathOperator\End{End}
\DeclareMathOperator\Aut{Aut}
\DeclareMathOperator\bik{bik}
\DeclareMathOperator\out{out}
\DeclareMathOperator\desc{desc}
\DeclareMathOperator\parb{par}

\DeclareMathOperator\id{id}

\DeclareMathOperator\bbN{\mathbb{N}}
\DeclareMathOperator\bbQ{\mathbb{Q}}

\DeclareMathOperator\bbZ{\mathbb{Z}}

\DeclareMathOperator\calP{\mathcal{P}}
\DeclareMathOperator\calL{\mathcal{L}}

\theoremstyle{definition}
\newtheorem{theorem}{Theorem}[section]
\newtheorem*{theorem*}{Theorem}
\newtheorem{lemma}[theorem]{Lemma}
\newtheorem*{lemma*}{Lemma}
\newtheorem{corollary}[theorem]{Corollary}
\newtheorem{corollary*}{Corollary}
\newtheorem{proposition}[theorem]{Proposition}
\newtheorem*{proposition*}{Proposition}

\newtheorem*{definition*}{Definition}
\newtheorem{remark}[theorem]{Remark}
\newtheorem*{remark*}{Remark}

\newtheorem*{example*}{Example}

\title{Willis Theory via Graphs}
\author{Timothy P. Bywaters \\ The University of Sydney \\ School of Mathematics and Statistics \\ NSW 2006, Australia \\ t.bywaters@maths.usyd.edu.au \and Stephan Tornier \\ The University of Newcastle \\ School of Mathematical and Physical Sciences \\ University Drive, Callaghan, NSW 2308, Australia \\ stephan.tornier@newcastle.edu.au}
\date{\today}

\begin{document}

\maketitle
\begin{abstract}
We study the scale and tidy subgroups of an endomorphism of a totally disconnected locally compact group using a geometric framework. This leads to new interpretations of tidy subgroups and the scale function. Foremost, we obtain a geometric tidying procedure which applies to endomorphisms as well as a geometric proof of the fact that tidiness is equivalent to being minimizing for a given endomorphism. Our framework also yields an endomorphism version of the Baumgartner-Willis tree representation theorem. We conclude with a construction of new endomorphisms of totally disconnected locally compact groups from old via HNN-extensions.
\end{abstract}

\vspace{0.6cm}

\vspace{0.2cm}

\section{Introduction}

The study of totally disconnected locally compact (t.d.l.c.) groups has its origin in the desire to understand general locally compact groups, which are ubiquitous in mathematics. Any locally compact group $G$ may be viewed as an extension of its connected component $G_{0}$ by the totally disconnected quotient $G/G_{0}$.

Every connected locally compact group is an inverse limit of Lie groups by the solution of Hilbert's fifth problem due to Gleason \cite{Gle52}, Yamabe \cite{Yam53}, Montgomery--Zippin \cite{MZ52} and others. Therefore, totally disconnected locally compact groups have become the new focus.

\vspace{0.2cm}

Little was known beyond van Dantzig's existence theorem of an identity neighbourhood basis consisting of compact open subgroups \cite{vDa31} until Willis raised the hope for a general structure theory of t.d.l.c. groups through the introduction of new concepts around their automorphisms \cite{Wil94}, \cite{Wil01}. Aside from developing the structure theory of totally disconnected locally compact groups \cite{Wil04}, \cite{BW06}, \cite{Wil07}, \cite{BRW12}, this theory has had many unexpected applications, including to the fields of random walks and ergodic theory \cite{DSW06}, \cite{JRW96}, \cite{PW03}, arithmetic groups \cite{SW13} and Galois theory \cite{CH09}.

It is hence natural to search for the largest natural framework of Willis theory. In \cite{Wil15}, the foundations of Willis theory have been extended from automorphisms to endomorphisms. Some work has been done to generalize existing results accordingly, see \cite{BGT16}. The present article contributes to this scheme: We characterize tidy subgroups and the scale function for endomorphisms of t.d.l.c. groups via graphs and permutations and thereby generalize work of M{\"o}ller \cite{Moe02} for the case of automorphisms.
We use these characterizations to give alternate proofs of fundamental results concerning the scale and tidy subgroups. The reader is referred to Section~\ref{sec:preliminaries} for an introduction to Willis theory, including the concepts of tidy subgroups and the scale function, as well as fundamentals about directed graphs.

\vspace{0.2cm}
Let $G$ be a t.d.l.c. group. Further, let $\alpha$ be a continuous endomorphism of $G$ and $U$ a compact open subgroup of $G$. Using a certain graph associated to the data $(G,\alpha,U)$ we give a geometric proof of existence of a subgroup of $U$ which is tidy above for $\alpha$ (\cite[Proposition 3]{Wil15}), as well as the tidiness below condition (\cite[Proposition 8]{Wil15}). Combining both yields the following characterization of the scale and tidiness, resembling \cite[Lemma 3.1]{Moe02} and \cite[Theorem 3.4]{Moe02}, see Lemma~\ref{lem:finite_vertices_case} and Theorem \ref{thm:tidy_tree}.

For $i\!\in\!\bbN_{0}$, define $v_{-i}\!:=\!\alpha^{-i}(U)\!\in\!\calP(G)$ and a rooted, directed graph $\Gamma_{+}$ by
\begin{displaymath}
  V(\Gamma_{+})\!=\!\{uv_{-i}\mid u\in U,\ i\in\bbN_{0}\},\quad E(\Gamma_{+})\!=\!\{(uv_{-i},uv_{-i-1})\mid u\in U,\ i\in\bbN_{0}\}.
\end{displaymath}

\begin{theorem*}
Let $G$ be a t.d.l.c. group, $\alpha\in\End(G)$ and $U\le G$ compact open.
\begin{itemize}[leftmargin=1cm]
 \item[(i)] If $\{v_{-i}\mid i\in\bbN_{0}\}$ is finite, then there is a compact open subgroup $U$ of $G$ with $\alpha(U)\le U$ and which is tidy for $\alpha$ and $s(\alpha)=1$.
 \item[(ii)] If $\{v_{-i}\mid i\in\bbN_{0}\}$ is infinite, then $U$ is tidy for $\alpha$ if and only if the graph $\Gamma_{+}$ is a directed tree, rooted at $v_{0}$ with constant in-valency (excluding the root) equal to $1$ and constant out-valency. In this case, $s(\alpha)$ equals said out-valency.
\end{itemize}
\end{theorem*}

We use this theorem to establish a new, geometric tidying procedure for the case of endomorphisms, see Theorem \ref{thm:tidy_existence}. It features yet another graph defined in terms of the data $(G,\alpha,U)$ which admits an action of $U_{++}$, a fundamental subgroup of $G$ associated to $\alpha$ and $U$, see Section \ref{sec:Willis_theory}.
Most of the work goes into showing that this graph admits a quotient with a connected component isomorphic to a regular rooted tree. The stabilizer of its root turns out to be tidy for $\alpha$.

Theorem \ref{thm:tidy_existence} and associated constructions result in a geometric proof of the fact \cite[Theorem 2]{Wil15} that tidiness is equivalent to being minimizing, see Theorem~\ref{thm:tidy_minimizing}. Using the aforementioned ideas, we obtain a tree representation theorem for a certain natural subsemigroup of $\End(G)$ associated to $\alpha$, analogous to \cite[Theorem 4.1]{BW04} for the case of automorphisms.

Finally, we give a simple way to construct endomorphisms of non-compact t.d.l.c groups from certain endomorphisms of compact groups. 

\subsection{Structure}
In Section \ref{sec:preliminaries} we collect the necessary preliminaries about Willis theory, graphs and the permutation topology which appears in this context. We then move on to characterize the notions of being tidy above, tidy below and tidy in terms of directed graphs in Section \ref{sec:char_tidy_subgroups}. These characterizations are used in Section \ref{sec:tidying_procedure} to develop a geometric tidying procedure. In return, this tidying procedure is employed in Section \ref{sec:scale_function} to recover fundamental results about the scale function. Section \ref{sec:tree_rep_thm} contains the tree representation theorem. Finally, Section \ref{sec:new_from_old} contains a construction of new endomorphisms of t.d.l.c. groups from old.

\subsection{Acknowledgments}
The authors owe thanks to George Willis for suggesting a graph-theoretic study of the scale and tidiness for endomorphisms in the spirit of M{\"o}ller, and for later discussions surrounding this work. Both authors also thank Jacqui Ramagge for helpful conversations and suggesting improvements on an early draft. This research was conducted whilst the first author was supported by the Australian Government Research Training Program Scholarship. The second author thanks The University of Newcastle and The University of Sydney for their hospitality, and gratefully acknowledges financial support through the SNSF Doc.Mobility fellowship 172120.

Finally, the authors are indebted to an anonymous referee whose comments not only significantly improved readability but also pointed out and suggested solutions to shortfalls in a number of arguments.

\section{Preliminaries}\label{sec:preliminaries}

In this section, we collect the necessary prerequisites for Willis theory, graphs and the permutation topology. Given a t.d.l.c. group $G$, we denote by $\End(G)$ the semigroup of continuous homomorphisms from $G$ to itself.

\subsection{Willis Theory}\label{sec:Willis_theory}
Throughout, $G$ denotes a t.d.l.c. group, $\alpha\in\End(G)$ an endomorphism of $G$ and $U\le G$ a compact open subgroup. We begin by defining essential subgroups associated to this data.

Set $U_{0}:=U$ and define inductively for $n\in\bbN_{0}$:
\begin{displaymath}
U_{n+1}:=U\cap\alpha(U_{n})\quad \quad\text{and}\quad U_{-n-1}:=U_{-n}\cap\alpha^{-n-1}(U)=\bigcap_{k = 0}^{n+1}\alpha^{-k}(U).
\end{displaymath}
Using these, we further define the following subgroups:
\begin{displaymath}
 U_{+}:=\bigcap_{n\in \bbN}U_{n},\quad U_{++}:=\bigcup_{n\in\bbN}\alpha^{n}(U_{+}),\quad U_{-}:=\bigcap_{n\in \bbN}U_{-n},\quad U_{--}:=\bigcup_{n\in\bbN}\alpha^{-n}(U_{-}).
\end{displaymath}
Both $(U_{n})_{n\in\bbN_{0}}$ and $(U_{-n})_{n\in\bbN_{0}}$ are descending sequences of compact subgroups of $U$. Whereas $U_{-n}$
is open for all $n\in\bbN_{0}$, the groups $U_{n}$ $(n\in\bbN)$ need not be so.

\begin{remark}\label{rem:comparison_of_definitions}
In the setting of automorphisms, the definitions of $U_{n+1}$ and $U_{+}$ given here agree with the definitions $U_{n+1} := \bigcap_{k = 0}^{n+1}\alpha^k(U)$ and $U_{+} := \bigcap_{k\ge 0} \alpha^k(U)$ which were originally given in \cite{Wil94}. They do not agree in the case of endomorphisms. See \cite[Section 3]{Wil15} for a discussion surrounding this. 
\end{remark}

There is an important alternative description of the above subgroups in terms of trajectories and regressive sequences: Let $x\in G$. The \emph{$\alpha$-trajectory} of $x$ is the sequence $(\alpha^{n}(x))_{n\in\bbN_{0}}$ in $G$. An \emph{$\alpha$-regressive trajectory} of $x$ is a sequence $(x_{n})_{n\in\bbN_{0}}$ in $G$ such that $x_{0}=x$ and $\alpha(x_{n})=x_{n-1}$ for all $n\in\bbN$. Consequently, we have the following verbal descriptions of $U_{\pm}$ and $U_{\pm\pm}$. See \cite[Proposition 1]{Wil15} for details.
\begin{displaymath}
 U_{+}=\left\{\text{\begin{tabular}{c} elements of $U$ which admit an \\ $\alpha$-regressive trajectory contained in $U$\end{tabular}}\right\},
\end{displaymath}
\begin{displaymath}
 U_{++}=\left\{\text{\begin{tabular}{c} elements of $G$ which admit an $\alpha$-regressive \\ trajectory that is eventually contained in $U$\end{tabular}}\right\},
\end{displaymath}
\begin{displaymath}
 U_{-}=\left\{\text{\begin{tabular}{c} elements of $U$ whose \\ $\alpha$-trajectory is contained in $U$\end{tabular}}\right\},
\end{displaymath}
\begin{displaymath}
 U_{--}=\left\{\text{\begin{tabular}{c} elements of $G$ whose $\alpha$-trajectory \\ is eventually contained in $U$\end{tabular}}\right\}.
\end{displaymath}

The following properties of compact open subgroups in relation to a given endomorphism lie at the heart of Willis theory.

The subgroup $U$ is \emph{tidy above} for $\alpha$ if $U = U_{+}U_{-}$, and \emph{tidy below} for $\alpha$ if $U_{--}$ is closed. It is \emph{tidy} for $\alpha$ if it is both tidy above and tidy below for $\alpha$.

\begin{remark}
Our definition of being tidy below deviates from \cite[Definition~9]{Wil15}. In the approach given there, the natural definition for being tidy below is that $U_{++}$ be closed and the sequence $([\alpha^{n+1}(U_+):\alpha^{n}(U_+)])_{n\in\bbN}$ be constant. Assuming tidiness above, the two definitions are equivalent by \cite[Proposition 9]{Wil15}. However, the definition given here is the one which appears naturally in our geometric context.
\end{remark}

The \emph{scale} of $\alpha$ is 
\[s(\alpha) = \min\{[\alpha(U):U\cap \alpha(U)]\mid U\le G \hbox{ compact open}\}.\]
A compact open subgroup $U\le G$ is \emph{minimizing} for $\alpha$ if $[\alpha(U):\alpha(U)\cap U]=s(\alpha)$.

\begin{remark}
The scale is well-defined as the index of an open subgroup inside a compact subgroup is finite: The latter can be covered by cosets of the open subgroup.
\end{remark}

One of the major achievements of \cite{Wil94},\cite{Wil01} and \cite{Wil15} is a proof showing that being minimizing and being tidy for a given endomorphism are equivalent conditions on compact open subgroups \cite[Theorem 2]{Wil15}. It relies on an algorithm, known as a \emph{tidying procedure}, which transforms, in a finite number of steps, an arbitrary compact open subgroup of $G$ into one that is tidy for a given endomorphism. This algorithm  allows one to prove many desirable properties of the scale function.

In \cite{Moe02}, a geometric interpretation of the scale and tidy subgroups as well as a new tidying procedure were given for the case of automorphisms. In the present work we generalize this framework to the setting of endomorphisms.

For the reader's convenience, we include the statement of \cite[Lemma 2]{Wil15} below. It constitutes an endomorphism version of the equality \begin{displaymath}
\alpha^{k}\left(\bigcap_{i=m}^{n}\alpha^{i}(U)\right)=\bigcap_{i=m+k}^{n+k}\alpha^{i}(U)
\end{displaymath}
which holds for an automorphism $\alpha$ and $k,m,n\in\bbZ$ with $m\le n$.

\begin{lemma}[{\cite[Lemma 2]{Wil15}}]\label{lem:wil_lem2}
Retain the above notation. For all $n,m\in\bbN_0$:
\begin{enumerate}[(i)]
  \item $U_{-n-m}=(U_{-n})_{-m}$,
  \item $\alpha^{k}(U_{-n})=\begin{cases} U_{k}\cap U_{k-n} & 0\le k\le n \\ \alpha^{k-n}(U_{n}) & k\ge n\end{cases}$, and
  \item\label{item:wil_lem2_iii} $(U_{-n})_{k}=U_{k}\cap U_{-n}$ for all $k\ge 0$ and $(U_{-n})_{+}=U_{+}\cap U_{-n}$.
\end{enumerate}
\end{lemma}

\subsection{Directed and Undirected Graphs}

We largely follow M{\"o}ller's notation for directed graphs. A \emph{directed graph $\Gamma$} is a tuple $(V(\Gamma),E(\Gamma))$ consisting of a \emph{vertex set} $V(\Gamma)$ and an \emph{edge set} $E(\Gamma)\subseteq V(\Gamma)\times V(\Gamma)\setminus\{(u,u)\mid u\in V(\Gamma)\}$. We denote by $o,t:E(\Gamma)\to V(\Gamma)$ the projections onto the first and second factor, the \emph{origin} and \emph{terminus} of an edge.
Let $\Gamma$ be a directed graph. An \emph{arc} of length $k\in \bbN$ from $v\in V(\Gamma)$ to $v'\in V(\Gamma)$ is a tuple $(v=v_{0},\ldots,v_{k}=v')$ of distinct vertices of $\Gamma$ such that $(v_{i},v_{i+1})$ in an edge in $\Gamma$ for all $i\!\in\!\{0,\ldots,k-1\}$. Two vertices $v,w\in\Gamma(V)$ are \emph{adjacent} if either $(v,w)\in E(\Gamma)$ or $(w,v)\in E(\Gamma)$. A \emph{path}  of length $k\in \bbN$ from $v\in V(\Gamma)$ to $v'\in V(\Gamma)$ is a tuple $(v=v_{0},\ldots,v_{k}=v')$ of distinct vertices of $\Gamma$ such that either $(v_{i},v_{i+1})$ or $(v_{i+1},v_{i})$ is an edge in $\Gamma$ for all $i\in\{0,\ldots,k-1\}$. The directed graph $\Gamma$ is connected if for all $v,w\in V(\Gamma)$ there is a path from $v$ to $w$. It is a \emph{tree} if it is connected and has no non-trivial cycles, i.e. tuples $(v_{0},\ldots,v_{k})$ with $k\ge 3$ and such that $(v_{0},\ldots,v_{k-1})$ and $(v_{k-1},v_k)$ are both paths and $v_{k}=v_{0}$. Two infinite paths in $\Gamma$ are \emph{equivalent} if they intersect in an infinite path. When $\Gamma$ is a tree, this is an equivalence relation on infinite paths and the \emph{boundary} $\partial\Gamma$ of $\Gamma$ is the set of these equivalence classes. For the following, let $v\!\in\! V(\Gamma)$. Set $\mathrm{in}_{\Gamma}(v)\!:=\!\{w\in V(\Gamma)\mid (w,v)\in E(\Gamma)\}$ and $\mathrm{out}_{\Gamma}(v):=\{w\in V(\Gamma)\mid (v,w)\in E(\Gamma)\}$. The \emph{in-valency} of $v\in V(\Gamma)$ is the cardinality of $\mathrm{in}_{\Gamma}(v)$ and the \emph{out-valency} of $v\in V(\Gamma)$ is the cardinality of $\mathrm{out}_{\Gamma}(v)$. The directed graph $\Gamma$ is \emph{locally finite} if all its vertices have finite in- and out-valency and \emph{regular} if $|\mathrm{in}_{\Gamma}(v)|$ and $|\out_{\Gamma}(v)|$ are constant for $v\in V(\Gamma)$.

A \emph{directed line} in $\Gamma$ is a sequence $(v_{i})_{i\in\bbZ}$ of distinct vertices such that either $(v_{i},v_{i+1})$ is an edge for every $i\in\bbZ$, or $(v_{i},v_{i-1})$ is an edge for every $i\in\bbZ$.

For a subset $A\subseteq V(\Gamma)$, the \emph{subgraph of $\Gamma$ spanned by $A$} is the directed graph with vertex set $A$ and edge set $\{(v,w)\in E(\Gamma)\mid v,w\in A\}$.

The \emph{set of descendants} of $v\!\in\! V(\Gamma)$ is $\mathrm{desc}_{\Gamma}(v)\!:=\!\{w\!\in\! V(\Gamma)\!\mid\!\exists\text{ arc from $v$ to $w$}\}$. For $A\subseteq V(\Gamma)$, set $\mathrm{desc}_{\Gamma}(A):=\bigcup_{v\in A}\mathrm{desc}_{\Gamma}(v)$. A directed tree $\Gamma$ is \emph{rooted} at $v_0\in V(\Gamma)$ if $\Gamma = \desc_{\Gamma}(v_0)$. In this case, $|\operatorname{in}_{\Gamma}(v)| = 1$ for all vertices $v\neq v_0$ and $|\operatorname{in}_{\Gamma}(v_0)|= 0$. The definition of being regular is altered for directed rooted trees: A directed tree rooted at $v_0$ is \emph{regular} if $|\out_\Gamma(v)|$ is constant for $v\in V(\Gamma)$.

A \emph{morphism} between directed graphs $\Gamma_{1}=(V_{1},E_{1})$ and $\Gamma_{2}=(V_{2},E_{2})$ is a pair $(\alpha_{V},\alpha_{E})$ of maps $\alpha_{V}:V_{1}\to V_{2}$ and $\alpha_{E}:E_{1}\to E_{2}$ preserving the graph structure, i.e. $\alpha_{V}(o(e))=o(\alpha_{E}(e))$ and $\alpha_{V}(t(e))=t(\alpha_{E}(e))$ for all $e\in E_{1}$. An \emph{automorphism} of a directed graph $\Gamma=(V,E)$ is a morphism $\alpha=(\alpha_{V},\alpha_{E})$ from $\Gamma$ to itself such that $\alpha_{V}$ and $\alpha_{E}$ are bijective and $\alpha$ admits an inverse morphism.

\vspace{0.2cm}
An \emph{undirected graph} $\Gamma$, after Serre \cite{Ser03}, is a tuple $(V,E)$ consisting of a \emph{vertex set} $V$ and an \emph{edge set} $E$, together with a fixed-point-free involution of $E$, denoted by $e\mapsto\overline{e}$, and maps $o,t:E\to V$, providing the \emph{origin} and \emph{terminus} of an edge, such that $o(\overline{e})=t(e)$ and $t(\overline{e})=o(e)$ for all $e\in E$. Given $e\in E$, the pair $\{e,\overline{e}\}$ is a \emph{geometric edge}. For $x\in V$, we let $E(x):=o^{-1}(x)=\{e\in E\mid o(e)=x\}$ be the set of edges issuing from $x$. The \emph{valency} of $x\in V$ is $|E(x)|$. In particular, a directed graph $(V,E)$ yields an undirected graph by passing to the symmetric closure of $E\subseteq V\times V$. The valency of a vertex in this undirected graph is the sum of the out- and in-valency of the same vertex in the directed graph. Undirected trees and their boundary are defined analogous to the case of directed trees. So are morphisms and automorphisms.

\subsection{Permutation Topology}\label{sec:perm_top}
The results of this article involve groups acting on directed graphs. In this situation, the acting group may be equipped with the \emph{permutation topology} for its action on the vertex set of said graph. This topology is naturally introduced in the following, more general setting. See e.g. \cite{Moe10}.

Let $X$ be a set and consider $G\le\Sym(X)$. The basic open sets for the permutation topology on $G$ are
\begin{displaymath}
 U_{x,y}:=\{g\in G\mid \forall i\in\{1,\ldots,n\}:\ g(x_{i})=y_{i}\}
\end{displaymath}
with $n\in\bbN$ and $x=(x_{1},\ldots,x_{n}), y=(y_{1},\ldots,y_{n})\in X^{n}$. The permutation topology turns $G$ into a topological group which is Hausdorff and totally disconnected. It  is locally compact if and only if the orbits of $U_{x}:= U_{x,x}$ are finite for all $x\in X$. Also, it makes the action map $G\times X\!\to\! X$ given by $(g,x)\!\mapsto\! g(x)$ continuous.

\section{Characterization of Tidy Subgroups}\label{sec:char_tidy_subgroups}

Let $G$ be a totally disconnected, locally compact group and let $\alpha\in\End(G)$. In this section, we characterize the compact open subgroups $U$ of $G$ which are tidy for $\alpha$ in terms of certain directed graphs. In doing so we generalize several results of \cite{Moe02} from conjugation automorphisms to general endomorphisms.

Frequently, we restrict to the case where the set $\{\alpha^{-i}(U)\mid i\in\bbN_{0}\}$ is infinite and hence all $\alpha^{-i}(U)$ ($i\in\bbN_{0}$) are distinct. The finite case corresponds to M{\"o}ller's periodicity case \cite[Lemma 3.1]{Moe02} and is covered by the following lemma.

\begin{lemma}\label{lem:finite_vertices_case}
Let $G$ be a t.d.l.c. group, $\alpha\in\End(G)$ and $U\le G$ compact open. If $\{\alpha^{-i}(U)\mid i\in\bbN_{0}\}$ is finite, then there is $N\in\bbN_{0}$ such that $\smash{V:=\bigcap_{k=0}^{N}\alpha^{-k}(U)}=U_{-}$ satisfies $\alpha(V)\le V$ and is tidy for $\alpha$.
\end{lemma}

\begin{proof}
If $\{\alpha^{-i}(U)\mid i\in\bbN_{0}\}$ is finite, then  $U_{-}=\bigcap_{k\in\bbN_{0}}\alpha^{-i}(U)$ is an intersection of finitely many open subgroups. Say $\smash{U_{-}\!=\!\bigcap_{k=0}^{N}\alpha^{-k}(U)\!=:\!V}$. Then $V\le G$ is compact open and $\alpha(V)\le V$. We conclude $V=V_{-}$. Hence $V$ is tidy above for $\alpha$. Since $V=V_{-}\le V_{--}$ we also deduce that $V_{--}$ is open and hence closed. Thus $V$ is also tidy below for $\alpha$.
\end{proof}

\subsection{Tidiness Above}\label{sec:tidy_above}
We recover the fact that for every compact open subgroup $U\le G$ there is $n\!\in\!\bbN_{0}$ such that $U_{-n}\!=\bigcap_{k=0}^{n}\alpha^{-n}(U)$ is tidy above for $\alpha$.

Consider the graph $\Gamma$ defined as follows: Set $v_{-i}:=\alpha^{-i}(U)\in\calP(G)$ for $i\in\bbN_0$, where $\calP(G)$ denotes the power set of $G$. Now set
\begin{displaymath}
V(\Gamma):=\{gv_{-i}\mid g\in G,\ i\in\bbN_{0}\} \quad\text{and}\quad E(\Gamma):=\{(gv_{-i},gv_{-i-1})\mid g\in G,\ i\in\bbN_{0}\}.
\end{displaymath}
Note that $G$ acts on $\Gamma$ by automorphisms via left multiplication. For this action, we compute the stabilizer $G_{v_{-i}}=\alpha^{-i}(U)$ $(i\ge 0)$, as well as
\begin{displaymath}
G_{\{v_{-m}\mid m\ge 0\}}=\bigcap\nolimits_{m\ge 0}\alpha^{-m}(U)=U_{-}.
\end{displaymath}

We now reprove \cite[Lemma 4]{Wil15} in terms of the graph $\Gamma$.

\begin{lemma}\label{lem:wil_lem4_graph}
Retain the above notation. Suppose that $U_{N}v_{-1}=U_{+}v_{-1}$ for some $N\in\bbN_0$. Then $U_{-n}v_{-n-1}=(U_{-n})_{+}v_{-n-1}$ for all $n\ge N$.
\end{lemma}

\begin{proof}
By definition, $(U_{-n})_{+}v_{-n-1}\subseteq U_{-n}v_{-n-1}$. Now, let $w\in U_{-n}v_{-n-1}$. Then there is $u\in U_{-n}$ such that $w=uv_{-n-1}$. We obtain $\alpha^{n}(u)\in\alpha^{n}(U_{-n})$ which equals $U_{n}$ by Lemma \ref{lem:wil_lem2} and is contained in $U_{N}$ since $n\ge N$. Hence, by assumption, there is $u_{+}\in U_{+}$ such that $\alpha^{n}(u)v_{-1}=u_{+}v_{-1}$. By the characterisation of $U_+$ as elements of $U$ admitting an $\alpha$-regressive trajectory contained in $U$, we have $\alpha^n(U_+\cap U_{-n}) = U_+$. Hence we may pick $u_{+}'\in U_{+}\cap U_{-n}$ with $u_{+}v_{-1}=\alpha^{n}(u_{+}')v_{-1}$. Then $u_{+}'\in (U_{-n})_{+}$ as by Lemma \ref{lem:wil_lem2} we have $U_{+}\cap U_{-n}=(U_{-n})_{+}$. We conclude that $u_{+}'v_{-n-1}=uv_{-n-1}$ since $u^{-1}u_{+}'\in U_{-n-1}\le G_{v_{-n-1}}$ by the following argument: We have $u^{-1}u_{+}'\in U_{-n}$ by definition, thus it suffices to show $u^{-1}u_{+}'\in\alpha^{-n-1}(U)$. By construction, $\alpha^{n}(u)^{-1}u_{+}\in G_{v_{-1}}=\alpha^{-1}(U)$, hence
\begin{displaymath}
\alpha^{n+1}(u^{-1}u_{+}')=\alpha^{n+1}(u^{-1})\alpha^{n+1}(u_{+}')=\alpha(\alpha^{n}(u^{-1})u_{+})\in U.
\end{displaymath}
\end{proof}

We use the following lemma to prove analogues of Theorems 2.1 and 2.3 in~\cite{Moe02}.

\begin{lemma}\label{lem:moe_thm2.3_end}
Retain the above notation. Fix $N\in\bbN_0$ and consider the following:
\begin{itemize}
  \item[(i)] $U_{N}v_{-1}=U_{+}v_{-1}$.
  \item[(ii)] For every $u\in U_{-N}$ there is $u_{+}\in U_{+}\cap U_{-N}$ with $u_{+}v_{i}=uv_{i}$ for all $i\le 0$.
  \item[(iii)] The subgroup $U_{-N}$ is tidy above for $\alpha$.
\end{itemize}
Then (i) implies (ii), and (ii) implies (iii).
\end{lemma}

\begin{proof}
To see (i) implies (ii) let $u\in U_{-N}$. By induction, we construct a sequence $\smash{(u_{n})_{n\in\bbN}}$ contained in $U_{+}\cap U_{-N}$ such that $\smash{u_{n}v_{i}=uv_{i}}$ for all $i\in\{-N-n,\ldots,0\}$. Then, as $U_{+}\cap U_{-N}$ is compact, $(u_{n})_{n\in\bbN}$ has an accumulation point $u_{+}\in U_{+}\!\cap\! U_{-N}$. We conclude that for any given $n\in\bbN$, we have
\begin{displaymath}
u_{k}^{-1}u_{+}\in G_{v_{-n}}=\alpha^{-n}(U)
\end{displaymath}
for large enough $k\in\bbN$ because $\alpha^{-n}(U)$ is open. That is, given $n\in\bbN$ we have 
\[u_+v_{-n} = u_kv_{-n} = uv_{-n}.\] 
for sufficiently large $k\in\bbN$. Now, by (i), Lemma \ref{lem:wil_lem4_graph} and Lemma \ref{lem:wil_lem2}\ref{item:wil_lem2_iii}, we may pick $\smash{u_{1}\in U_{+}\cap U_{-N}}$ such that $u_{1}v_{-N-1} = uv_{-N-1}$. Next, assume that $\smash{u_{n}}$ has been constructed for some $n\in\bbN$. Then $u_n^{-1}uv_i = v_i$ for all $i\in\{-N-n,\ldots, 0\}$. That is,
\begin{displaymath}
  u_{n}^{-1}u\in \bigcap_{i = 0}^{n + N}\alpha^{-i}(U) = U_{-N-n}.
\end{displaymath}
By Lemma \ref{lem:wil_lem4_graph}, there exists $x\in (U_{-N-n})_{+}$ such that $u_n^{-1}uv_{-N-n-1} = xv_{-N-n-1}$. By assumption, $u_{n}\in U_{+}\cap U_{-N}$ and, by Lemma \ref{lem:wil_lem2}, $x\in(U_{-N-n})_{+}\!=\!U_{+}\cap U_{-N-n}$. Hence $u_{n}x\in U_{+}\cap U_{-N}$. Also, $u_nx(v_i) = u(v_i)$ for all $i\in\{-N-n-1,\ldots, 0\}$. We may therefore set $u_{n+1}:=u_{n}x$.

To see that (ii) implies (iii) we use that, by assumption, for every $u\in U_{-N}$ there is $u_{+}\in U_{+}\cap U_{-N}$ such that $u$ and $u_{+}$ agree on $v_{i}$ for all $i\le 0$. Set $u_{-}:=u_{+}^{-1}u$. Then $u_{-}v_{i}=v_{i}$ for all $i\le 0$. Hence $u_{-}\in G_{\{v_{m}\mid m\le 0\}}=U_{-}$ and
\begin{displaymath}
U_{-N}=(U_{+}\cap U_{-N})U_{-} = (U_{-N})_+(U_{-N})_{-}
\end{displaymath}
by Lemma \ref{lem:wil_lem2} as required.
\end{proof}

\begin{theorem}\label{thm:tidy_above_exists}
Let $G$ be a t.d.l.c. group, $\alpha\in\End(G)$ and $U\le G$ compact open. Then there is $N\in\bbN$ such that $U_{N}v_{-1} = U_{+}v_{-1}$, and $U_{-N}$ is tidy above for $\alpha$.
\end{theorem}

\begin{proof}
First note that $U_{+}v_{-1}\subseteq U_{m}v_{-1}\subseteq U_{n}v_{-1}$ for all $0\le n\le m$ since the sets $U_{n}$ ($n\in\bbN_{0}$) are nested. Thus it suffices to show that $U_{N}v_{-1}\subset U_{+}v_{-1}$ for some $N\in\bbN$. Towards a contradiction, assume that $U_{+}v_{-1}\subsetneq U_{n}v_{-1}$ for all $n\in\bbN$,
i.e. there is $w_{n}\in V(\Gamma)$ such that $w_{n}\in U_{n}v_{-1}$ for all $n\in\bbN$ but $w_{n}\not\in U_{+}v_{-1}$. Then there is a sequence  $(u_{n})_{n\in\bbN}$ contained in $U$ such that $u_{n}\in U_{n}$ and $u_{n}v_{-1}=w_{n}$. Since $U$ is compact, the sequence $(u_{n})_{n\in\bbN}$ has an accumulation point $u_{+}$ in $U$. This accumulation point has to be contained in $U_{+}$: Indeed, pick a subsequence $(u_{n_{k}})_{k\in\bbN}$ of $(u_{n})_{n\in\bbN}$ converging to $u_{+}$. Then for any given $m\in\bbN$, we have $u_{n_{k}}\in U_{m}$ for almost all $k$. Since $U_{m}$ is closed we conclude that $u_{+}\in U_{m}$ for every $m\in\bbN$. Hence
\begin{displaymath}
 u_{+}\in\bigcap_{m\in\bbN}U_{m}=U_{+}.
\end{displaymath}
Furthermore, if $u_{+}v_{-1}=w$, then because $u_{+}u_{n_{k}}^{-1}$ is contained in the open set $G_{v_{-1}}$ for large enough $k\in\bbN$ we must have $w = w_k$ for sufficiently large $k\in\bbN$. We conclude that $w_k\in U_{+}v_{-1}$ for sufficiently large $k\in\bbN$ and thus we have a contradiction. Now, $U_{-N}$ is tidy above for $\alpha$ by Lemma \ref{lem:moe_thm2.3_end}.
\end{proof}

\begin{theorem}\label{thm:tidy_above}
Let $G$ be a t.d.l.c. group, $\alpha\in\End(G)$ and $U\le G$ compact open. Then the following statements are equivalent:
\begin{enumerate}[(i)]
  \item\label{item:thm_tidy_above_cosets} $UU_{-1} = U_+U_{-1}$.
  \item $Uv_{-1}=U_{+}v_{-1}$.
  \item\label{item:thm_tidy_above_U_+_action} For every $u\in U$ there is $u_{+}\in U_{+}$ such that $u_{+}v_{i}=uv_{i}$ for all $i\le 0$.
  \item The subgroup $U$ is tidy above for $\alpha$.
\end{enumerate}
\end{theorem}
\begin{proof}
That (i) implies (ii) follows as $U_{-1}$ is a subgroup. That (ii) implies (iii) and (iii) implies (iv) follows from Lemma \ref{lem:moe_thm2.3_end} for $N = 0$.	
Finally, if (iv) holds, then $UU_{-1}=U_{+}U_{-}U_{-1}=U_{+}U_{-1}$ as $U_{-}\le U_{-1}$.
\end{proof}

\begin{proposition}\label{prop:out_index}
Let $G$ be a t.d.l.c. group, $\alpha\in\End(G)$ and $U\le G$ compact open as well as tidy above for $\alpha$. Then 
\begin{align*}
[\alpha^{-n}(U):\alpha^{-n-1}(U)\cap\alpha^{-n}(U)] & =[\alpha^{-n}(U)\cap U:\alpha^{-n-1}(U)\cap\alpha^{-n}(U)\cap U] \\
& = [U_{-n}:U_{-n-1}]\\
& = [U:U_{-1}]
\end{align*}
for all $n\in\bbN$.
\end{proposition}
\begin{proof}
We prove the equality $[U_{-n}:U_{-n-1}]=[U:U_{-1}]$. The others are proven analogously. The homomorphism $\alpha^{n}$ induces an injective map from $U_{-n}/U_{-n-1}$ to $U/U_{-1}$ as $(\alpha^{n})^{-1}(U_{-1})\cap U_{-n}=U_{-n-1}$. This map is also surjective because $\alpha^{n}(U_{+}\cap U_{-n})=U_{+}$ and $U_{+}U_{-1}=UU_{-1}$ by Theorem \ref{thm:tidy_above}\ref{item:thm_tidy_above_cosets}.
\end{proof}

The following equality is used in Section \ref{sec:tree_rep_thm}.

\begin{lemma}\label{lem:coset_calculations}
Let $G$ be a t.d.l.c. group, $\alpha\in\End(G)$ and $U\le G$ compact open as well as tidy above for $\alpha$. Then $[\alpha(U):U\cap \alpha(U)] = [\alpha(U_+):U_+]$.
\end{lemma}

\begin{proof}
Note that 
\begin{displaymath}
  \alpha(U)(U\cap \alpha(U)) = \alpha(U_+)\alpha(U_-)(U\cap \alpha(U)) = \alpha(U_+)(U\cap \alpha(U))
\end{displaymath}
as $\alpha(U_-)\le \alpha(U\cap \alpha^{-1}(U))\le \alpha(U)\cap U$. Thus 
\[[\alpha(U):U\cap \alpha(U)] = [\alpha(U_+):U\cap \alpha(U)\cap \alpha(U_+)] = [\alpha(U_+):U\cap\alpha(U_+)].\]
Since $U\cap\alpha(U_+) = U_+$, the desired equality follows.
\end{proof}

\subsection{Tidiness Below}
In this section we present a geometric proof for the commonly used criterion that identifies a compact open and tidy above subgroup $U\le G$ as tidy below if we have $U_{--}\cap U=U_{-}$, cf. \cite[Proposition 8]{Wil15}.

First, recall that $U_{++}\!=\!\bigcup_{i\in\bbN_{0}}\alpha^{i}(U_{+})$ and $\smash{U_{--}\!=\!\bigcup_{i\in\bbN_{0}}\alpha^{-i}(U_{-})}$. In terms of the graph $\Gamma$ introduced in Section \ref{sec:tidy_above}, we have
\begin{displaymath}
U_{--}=\bigcup_{n\in\bbN}G_{\{v_{m}\mid m\le -n\}}.
\end{displaymath}

\begin{lemma}\label{lem:tidy_below}
Let $G$ be a t.d.l.c. group, $\alpha\in\End(G)$ and $U\le G$ be compact open as well as tidy above for $\alpha$. Then
\begin{enumerate}[(i)]
 \item the group $U_{--}\le G$ is closed if and only if $U_{--}\cap U=U_{-}$, and
 \item \label{item:lem:tidy_below:U_+} if $U_{--}$ is closed, then $U_{++}\cap U=U_{+}$.
\end{enumerate}
\end{lemma}

\begin{proof}
For (i), first assume that $U_{--}\cap U=U_{-}$. Then $U_{--}\cap U$ is closed. Since $U$ is open, this implies that $U_{--}$ is closed, see \cite[Chapter III.2.1, Proposition 4]{Bou98}.

Now suppose that $U_{--}\cap U\neq U_{-}$. By definition, $U_{-}\subseteq U_{--}\cap U$. Hence there exists $u\in U=G_{v_{0}}$ with $u\in G_{\{v_{m}\mid m\le -n\}}$ for some $n\in\bbN$ but $u\not\in U_{-}=G_{\{v_{m}\mid m\le 0\}}$. Then there is $l\in\bbN$ with $0<l<n$ and such that $uv_{-l}\neq v_{-l}$. Since $U$ is tidy above, we may decompose $u=u_{+}u_{-}$ for some $u_{+}\in U_{+}$ and $u_{-}\in U_{-}$. Hence, replacing $u$ with $uu_{-}^{-1}$, we may assume that $u\in U_{+}$. Choose an $\alpha$-regressive trajectory $(u_{j})_{j\in\bbN_0}$ of $u$ contained in $U_{+}$. Define a sequence $(x_{i})_{i\in\bbN_0}$ contained in $U_{--}\cap U_{+}\le U$ as follows: Set $x_{0}:=u$ and $x_{i+1}:=x_{i}u_{(i+1)n}$. We collect the relevant properties of the sequences $(u_{j})_{j\in\bbN_0}$ and $(x_{i})_{i\in\bbN_0}$ in the following lemma, see below for an illustration of the second sequence.

\begin{lemma}\label{lem:aux_sequences}
The sequences $(u_{j})_{j\in\bbN_0}$ and $(x_{i})_{i\in\bbN_0}$ have the following properties.
\begin{itemize}
 \item[(a)] For all $j\in\bbN_0$: $u_{j}\in G_{\{v_{m}\mid m\le -n-j\}}\cap G_{\{v_{m}\mid -j\le m\le 0\}}\cap U_{+}\le U_{--}\cap U_{+}$.
 \item[(b)] For all $i\in\bbN$: $x_{i-1}\in G_{\{v_{m}\mid m\le -in\}}\cap U_{+}\le U_{--}\cap U_{+}$.
 \item[(c)] For all $j\in\bbN_0$: $u_{j}\not\in G_{v_{-l-j}}$.
 \item[(d)] For all $i\in\bbN_0$ and $0\le j\le i$: $x_{i}\not\in G_{v_{-l-jn}}$ and $x_{i+1}v_{-l-jn}=x_{i}v_{-l-jn}$.
\end{itemize}
\end{lemma}

\begin{proof}
For (a), note that $\alpha^{j}(u_{j})=u\in G_{\{v_{m}\mid m\le -n\}}=\bigcap_{k\ge n}\alpha^{-k}(U)$ by assumption and therefore $\smash{u_{j}\in\alpha^{-j}\big(\bigcap_{k\ge n}\alpha^{-k}(U)\big)=\bigcap_{k\ge n+j}\alpha^{-k}(U)=G_{\{v_{m}\mid m\le -n-j\}}}$. For the second part, simply recall that $(u_{j})_{j}$ is an $\alpha$-regressive trajectory of $u$ contained in $U_{+}$; in particular, $u_{j}\in U_{+}$ and $\alpha^{m}(u_{j})\in U_{+}\le U$ for all $0\le m\le j$. Therefore, $u_{j}\in\alpha^{-m}(U)=G_{v_{-m}}$ for all $0\le m\le j$.

Part (b) follows from (a) given that $x_{i}=x_{i-1}u_{in}=uu_{n}\cdots u_{(i-1)n}u_{in}$.

For part (c), recall that we have $u\not\in\alpha^{-l}(U)=G_{v_{-l}}$ by assumption and therefore $u_{j}\not\in\alpha^{-l-j}(U)=G_{v_{-l-j}}$. In order to prove part (d), we argue by induction: The element $x_{0}=u$ satisfies $x_{0}\not\in G_{v_{-l}}$ by part (c). Also $x_{1}v_{-l}=x_{0}v_{-l}$ because $x_{0}^{-1}x_{1}=u^{-1}uu_{n}=u_{n}$ and $u_{n}\in G_{\{v_{m}\mid -n\le m\le 0\}}$ by part (a). Now assume the statement holds true for $i\in\bbN$ and consider $x_{i+1}=x_{i}u_{(i+1)n}$. Then $x_{i+1}\not\in G_{v_{-l-(i+1)n}}$ because $u_{(i+1)n}\not\in G_{v_{-l-(i+1)n}}$ by part (c) whereas $x_{i}\in G_{v_{-l-(i+1)n}}$ by part (b). Also, $x_{i+1}v_{-l-jn}=x_{i}v_{-l-jn}$ for all $0\le j\le i$ since $x_{i+1}=x_{i}u_{(i+1)n}$ and $u_{(i+1)n}\in G_{\{v_{m}\mid -(i+1)n\le m\le 0\}}$ by part (a).
\end{proof}

By Lemma \ref{lem:aux_sequences}, the sequence $(x_{i})_{i\in\bbN_{0}}\subseteq U_{--}\cap U_{+}\subseteq U$ has the following shape, analogous to \cite[Figure 1]{Moe02}.
\begin{figure}[ht]
\begin{displaymath}
  \includegraphics{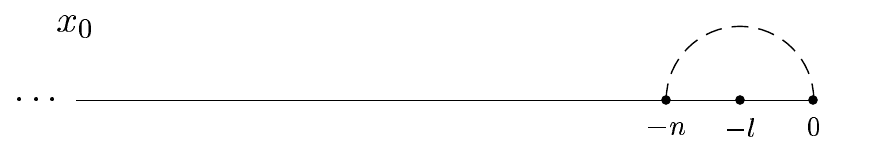}
\end{displaymath}
\begin{displaymath}
  \includegraphics{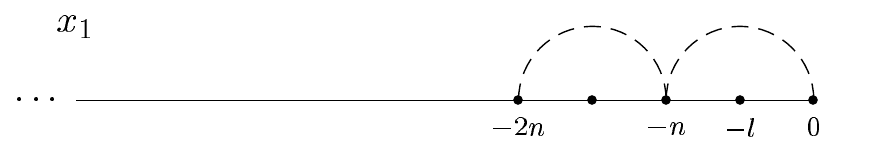}
\end{displaymath}
\begin{displaymath}
  \includegraphics{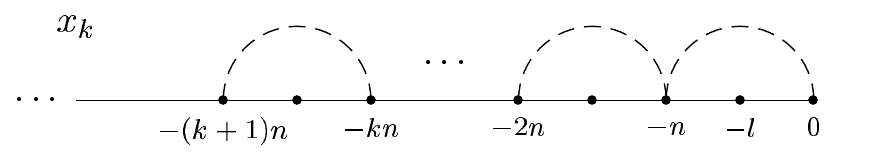}
\end{displaymath}
\begin{displaymath}
  \includegraphics{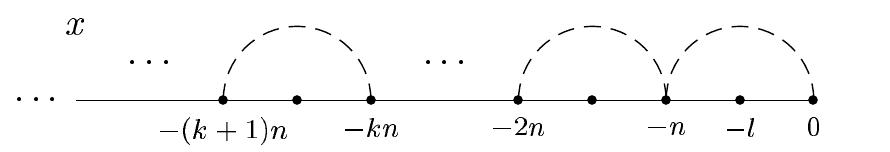}
\end{displaymath}
\caption{The sequence $(x_{i})_{i\in\bbN_{0}}$.}
\end{figure}

Now, since $U$ is compact, the sequence $(x_{i})_{i\in\bbN_{0}}\subseteq U_{--}\cap U_{+}\subseteq U$ has an accumulation point $x\in U$. However, $x\not\in U_{--}$ and hence $U_{--}$ is not closed.

For part (ii), note that $U_{+}\subseteq U_{++}\cap U$ by definition. Hence, towards proving the converse, we assume that there is $u\in(U_{++}\cap U)\backslash U_{+}$. Since $U$ is tidy above we may decompose $u=u_{+}u_{-}$ with $u_{+}\in U_{+}$ and $u_{-}\in U_{-}$. Replacing $u$ with $u_{+}^{-1}u\in(U_{++}\cap U)\backslash U_{+}$ we may hence assume $u\in U_{-}$.

Now, since $u\in U_{++}$, there is an $\alpha$-regressive trajectory $(u_{n})_{n\in\bbN_{0}}$ of $u$ in $G$ such that for some $N\in\bbN$ we have $u_{n}\in U_{+}$ for all $n\ge N$ and $u_{N-1}\not\in U$. Consider the element $u_{N}\in U$. For $n\ge N$ we have $\alpha^{n}(u_{N})=\alpha^{n-N}(u)\in U_{-}$. Hence $u_{N}\in U_{--}\cap U$. However, $u_{N}\not\in U_{-}$: Indeed, $u_{N}\not\in G_{v_{-1}}=\alpha^{-1}(U)$ because $u_{N-1}\not\in U$. Therefore, by part (i), $U_{--}$ is not closed.
\end{proof}

\subsection{Tidiness}\label{sec:tidiness}
We now combine the previous sections in order to characterize tidiness in terms of a subgraph of the graph $\Gamma$ introduced above. As before, let $G$ be a t.d.l.c. group, $\alpha\in\End(G)$ and $U\le G$ compact open.
Recall the definition $v_{-i}:=\alpha^{-i}(U)\in\calP(G)$ for $i\in\bbN_{0}$. We consider the subgraph $\Gamma_{+}$ of $\Gamma$ defined by
\begin{displaymath}
  V(\Gamma_{+}):=\{uv_{-i}\mid u\in U,\ i\in\bbN_{0}\}, \quad E(\Gamma_{+}):=\{(uv_{-i},uv_{-i-1})\mid u\in U,\ i\in\bbN_{0}\}.
\end{displaymath}
Note that the action of $U\le G$ on $\Gamma$ preserves $\Gamma_{+}\subseteq\Gamma$ and that $\Gamma_{+} = \desc_{\Gamma_{+}}(v_{0})$.

\begin{lemma}\label{lem:tidy_above_transitive}
Let $G$ be a t.d.l.c. group, $\alpha\in\End(G)$ and $U\le G$. If $U$ is tidy above for $\alpha$, then $U$ acts transitively on arcs of a given length issuing from $v_{0}\in V(\Gamma_{+})$.
\end{lemma}

\begin{proof}
Given that $|\mathrm{out}_{\Gamma_{+}}(v_{-n+1})|=[\alpha^{-n+1}(U)\cap U:\alpha^{-n+1}(U)\cap\alpha^{-n}(U)\cap U]$ and $U_{\{v_{-k}\mid k\le n-1\}}\!=\!U_{-n+1}$, the assertion follows by induction from Proposition~\ref{prop:out_index}.
\end{proof}

We are now ready to characterize tidiness of $U$ in terms of $\Gamma_{+}$ when the set $\{v_{-i}\mid i\in\bbN_{0}\}$ is infinite. Concerning the case where $\{v_{-i}\mid i\in\bbN_{0}\}$ is finite, Theorem \ref{thm:tidy_tree} is complemented by Lemma \ref{lem:finite_vertices_case}.

The following lemmas will help to identify vertices in $\Gamma_{+}$, and later on vertices in $\Gamma_{++}$, to be defined in Section \ref{sec:tidying_procedure}, as (un)equal and are frequently used in the proof of Theorem \ref{thm:tidy_tree}.
\begin{lemma}\label{lem:equal_vert_in_gamma}
Retain the above notation and suppose $\{v_{-i}\mid i\in\bbN_{0}\}$ is infinite. Let $g_0v_{-i}$, $g_{1}v_{-j}\in V(\Gamma)\subseteq\calP(G)$, where $g_0,g_1\in G$. If $g_{0}v_{-i}=g_{1}v_{-j}$, then $i=j$.
\end{lemma}

\begin{proof}
Since $\{v_{-i}\mid i\in\bbN_{0}\}$ is infinite and $v_{-i}:=\alpha^{-i}(U)$, its elements are pairwise distinct. The assertion hence follows from the fact that left cosets of distinct subgroups of a given group are distinct.
\end{proof}

Lemma \ref{lem:equal_vert_in_gamma}, applied to $\Gamma_{+}$ as a subgraph of $\Gamma$, shows that the depth of a vertex $uv_{-n}$ $(u\in U,\ n\in\bbN_{0})$, defined as the length of the shortest arc from the root to $uv_{-n}$, is given by $n$.

\begin{lemma}\label{lem:gamma_+_depth}
Retain the above notation and suppose $\{v_{-i}\mid i\in\bbN_{0}\}$ is infinite. Let $u\in U$ and $n\in\bbN$. Then $uv_{-n}\in V(\Gamma_{+})$ has depth $n$ in $\Gamma_{+}$.
\end{lemma}

\begin{proof}
Let $(w_{0}=v_{0},w_{1},\ldots,w_{k})$ be an arc from $v_{0}$ to $uv_{-n}$. By definition, we have $w_{1}=u_{1}v_{-1}$ for some $u_{1}\in U$. Hence, by induction, $w_{k}=u_{k}v_{-k}$ for some $u_{k}\in U$. But $w_{k}=uv_{-n}$. Therefore, $n=k$ by Lemma \ref{lem:equal_vert_in_gamma}.
\end{proof}

\begin{theorem}\label{thm:tidy_tree}
Let $G$ be a t.d.l.c. group, $\alpha\in\End(G)$ and $U\le G$ compact open. Assume $\{v_{-i}\mid i\in\bbN_{0}\}$ is infinite. Then $U$ is tidy for $\alpha$ if and only if $\Gamma_{+}$ is a directed tree with constant in-valency $1$, excluding $v_{0}$, as well as constant out-valency.
\end{theorem}
\begin{proof}
First, assume that $U$ is tidy for $\alpha$. Notice that for a given $i\le 0$, the in- and out-valency is constant among the collection of vertices $\{uv_{i}\mid u\in U\}$ given that $U$ acts on $\Gamma_{+}$ by automorphisms.

Concerning in-valencies it therefore suffices to show that each $v_{-i}$ for $i\ge 1$ has in-valency equal to one, note that Lemma \ref{lem:equal_vert_in_gamma} already implies that $|\mathrm{in}_{\Gamma_{+}}(v_0)| = 0$. Suppose otherwise, that is $|\mathrm{in}_{\Gamma_{+}}(v_{-i})|\ge 2$ for some $i\ge 1$. Then there is $u\in U_{v_{-i}}\subseteq\alpha^{-i}(U)$ such that $u v_{-i+1}\neq v_{-i+1}$. By Theorem \ref{thm:tidy_above} we may assume that $u\in U_{+}$. Now consider $u':=\alpha^{i}(u)\in U_{++}\cap U$. Since $U$ is tidy below, Lemma \ref{lem:tidy_below} shows that $u'\in U_{+}=U_{++}\cap U$. But $u\not\in\alpha^{-i+1}(U)$ and hence $u'=\alpha^{i}(u)\not\in\alpha(U)\supseteq U_{+}$, a contradiction. Thus $\Gamma_{+}$ is a directed tree.

Concerning out-valencies, we may also restrict our attention to $\{v_{-i}\mid i\in\bbN_{0}\}$. For $i\in\bbN$, we have $|\mathrm{out}_{\Gamma_{+}}(v_{-i})|=[\alpha^{-i}(U)\cap U:\alpha^{-i}(U)\cap U\cap\alpha^{-i-1}]$ which equals $[U:U_{-1}]=|\out_{\Gamma_{+}}(v_{0})|$ by Proposition \ref{prop:out_index}.

\vspace{0.2cm}
Now assume that $\Gamma_{+}$ has all the stated properties. Since $\Gamma_{+}$ is a tree, we have $U_{--}\cap U\subseteq U_{-}$ while the reverse inclusion holds by definition. Hence $U_{--}$ is closed by Lemma \ref{lem:tidy_below} and $U$ is tidy below. Combining the constant out-valency assumption with the fact that $\Gamma_{+}$ is a tree we obtain the equality $|Uv_{-1}|=|U_{-i}v_{-i-1}|$. Next, $|U_{-i}v_{-i-1}|=|U_{i}v_{-1}|$ since $|U_{i}v_{-1}|\le|Uv_{-1}|$ and due to the following observation: If $u\in U_{-i}$ is such that $uv_{-i-1}\neq v_{-i-1}$, then $\alpha^{i}(u)\in\alpha^{i}(U_{-i})=U_{i}$ by Lemma \ref{lem:wil_lem2} and $\alpha^{n}(u)v_{-1}\neq v_{-1}$. Thus $|U_{i}v_{-1}|\ge|U_{-i}v_{-i-1}|$. Overall, we conclude $|Uv_{-1}|=|U_{i}v_{-1}|$.

Finally, to see that the above implies $|Uv_{-1}|=|U_{+}v_{-1}|$, let $u\in U$. Then for every $i\in\bbN$ there is $u_{i}\in U_{i}$ with $uv_{-1}=u_{i}v_{-1}$. The sequence $(u_{i})_{i\in\bbN}$ is contained in $U$ and hence admits a convergent subsequence. Any such subsequence converges to an element $u_{+}\in\bigcap_{i\ge 0}U_{i}=U_{+}$ which is such that $u_{+}^{-1}u\in U_{v_{-1}}$. Theorem \ref{thm:tidy_above} now implies that $U$ is tidy above.
\end{proof}

The following Lemma is a useful test of tidiness as it relies only on calculating inverse images and indices. It is, in a sense, an algebraic way to see if $\Gamma_{+}$ satisfies the requirements of Theorem \ref{thm:tidy_tree}. We apply it multiple times in upcoming sections.

\begin{lemma}\label{lem:tidy_powers}
Let $G$ be a t.d.l.c. group, $\alpha\in\End(G)$ and $U\le G$ compact open. Then $U$ is tidy for $\alpha$ if and only if $[U\!:\!U\!\cap\!\alpha^{-n}(U)]\!=\![U\!:\!U\!\cap\!\alpha^{-1}(U)]^{n}$ for all $n\!\in\!\bbN$. In particular, if $U$ is tidy for $\alpha$, then $U$ is tidy for $\alpha^k$ for all $k\in\bbN$.
\end{lemma}

\begin{proof}
First, assume that $U$ is tidy for $\alpha$. If $\{v_{-i}\mid i\in\bbN_{0}\}$ is finite, then for some $N\in\bbN_{0}$ we have $[U_{-N}: U_{-N-1}] = 1$ by Lemma \ref{lem:finite_vertices_case}. Proposition \ref{prop:out_index} shows that $1=[U:U\cap\alpha^{-1}(U)]$ which implies $\alpha^{-1}(U)\supseteq U$. Therefore $\alpha^{-n}(U)\supseteq U$ for all $n\in\bbN$ and the both assertions follow. Now assume that $\{v_{-i}\mid i\in\bbN_{0}\}$ is infinite. Then $\Gamma_{+}$ is a rooted directed tree with constant out-valency $d$. The vertices at depth $n$ in this rooted directed tree are precisely the vertices $uv_{-n}$ $(u\in U)$ by Lemma \ref{lem:gamma_+_depth} and the definition of $\Gamma_{+}$. We obtain
\begin{displaymath}
  [U:U\cap\alpha^{-n}(U)]=|Uv_{-n}|=d^{n}=[U:U\cap\alpha^{-1}(U)]^{n}
\end{displaymath}
as desired. Conversely, assume that $[U:U\cap\alpha^{-n}(U)]=[U:U\cap\alpha^{-1}(U)]^{n}$ for all $n\in\bbN$ and consider the graph $\Gamma_{+}$. We have $d:=|\mathrm{out}_{\Gamma_{+}}(v_{0})|=[U:U\cap\alpha^{-1}(U)]$ as before. By definition of $\Gamma_{+}$, the out-valency of any other vertex is at most $d$. But
\begin{displaymath}
  |Uv_{-n}|=[U:U\cap\alpha^{-n}(U)]=[U:U\cap\alpha^{-1}(U)]^{n}=d^{n}
\end{displaymath}
by assumption. Since $Uv_{-n}$ is precisely the set of vertices at depth $n$ in $\Gamma_{+}$ by Lemma \ref{lem:gamma_+_depth}, we conclude that every vertex has out-valency $d$. Hence $\Gamma_{+}$ is a tree of constant in-valency $1$, excluding $v_{0}$, and $U$ is tidy for $\alpha$ by Theorem~\ref{thm:tidy_tree}.

For the final claim, note that if $U$ is tidy for $\alpha$, then $[U:U\cap \alpha^{-n}(U)] = d^n$ for all $n\in\bbN_{0}$. Thus, given $k\in\bbN$, we have
\[[U:U\cap \alpha^{-kn}(U)] = d^{kn} = (d^k)^n = [U:U\cap \alpha^{-k}(U)]^n.\] 
By the first claim wee see that $U$ is tidy for $\alpha^k$.
\end{proof}

\section{A Graph-Theoretic Tidying Procedure}\label{sec:tidying_procedure}

Let $G$ be a totally disconnected, locally compact group and let $\alpha\in\End(G)$. We show that there is a compact open subgroup of $G$ which is tidy for $\alpha$. 

The proof is algorithmic: Starting from an arbitrary compact open subgroup we construct a locally finite graph $\Gamma_{++}$. A certain quotient, inspired by \cite{Moe00}, of this graph has a connected component isomorphic to a regular rooted tree which admits an action of a subgroup of $G$. The stabilizer of the root in this tree is the desired tidy subgroup.

\vspace{0.2cm}
For the remainder of the section, fix $U\le G$ compact open. Referring to Lemma~\ref{lem:finite_vertices_case}, we shall assume throughout that $\{\alpha^{-i}(U)\mid i\in\bbN_{0}\}$ is infinite. By Theorem \ref{thm:tidy_above_exists} we may also assume that $U$ is tidy above for $\alpha$.

\subsection{The Graph $\Gamma_{++}$}
Consider the graph $\Gamma_{++}$ defined by
\[
  V(\Gamma_{++})\!=\!\{uv_{-i}\mid u\in U_{++}\hbox{, }i\in\bbN_0\}, \text{ and}
\]
\vspace{-0.45cm}
\[
  E(\Gamma_{++})\!=\!\{(uv_{-i},uv_{-i-1})\mid u\in U_{++}\hbox{, } i\in\bbN_0\}.
\]

The following remark will be used in the proof of Theorem \ref{thm:tidy_existence}.

\begin{remark}\label{rem:rel_gamma_++_gamma}
Note that $\Gamma_{++}$ is a subgraph of $\Gamma$. Also, if $U$ is tidy above for $\alpha$, the graphs $\Gamma_{+}$ and $\Gamma$ have the same out-valency by Proposition \ref{prop:out_index}, given that
\begin{align*}
|\out_{\Gamma}(v_{-i})|&=[\alpha^{-i}(U):\alpha^{-i-1}(U)\cap\alpha^{-i}(U)]\text{, and} \\
|\out_{\Gamma_{+}}(v_{-i})|&=[\alpha^{-i}(U)\cap U:\alpha^{-i-1}(U)\cap\alpha^{-i}(U)\cap U].
\end{align*}
Consequently, $\mathrm{desc}_{\Gamma}(v_{0})=\Gamma_{+}\subseteq \mathrm{desc}_{\Gamma_{++}}(v_{0})\subseteq\mathrm{desc}_{\Gamma}(v_{0})=\Gamma_{+}$
and therefore $\mathrm{desc}_{\Gamma_{++}}(v_{0})=\mathrm{desc}_{\Gamma}(v_{0})=\Gamma_{+}$.
\end{remark}

Note that $U_{++}$ acts on $\Gamma_{++}$ by automorphisms. We now define an injective graph endomorphism of $\Gamma_{++}$ that appears frequently. Let $uv_{i}\in V(\Gamma_{++})$ where $u\in U_{++}$. Since $\alpha(U_{++}) = U_{++}$, there exists $u'\in U_{++}$ such that $\alpha(u') = u$. Define $\rho(uv_{i}) = u'v_{i-1}$. The following proposition summarizes the properties of $\rho$ and includes justification that $\rho$ is well-defined.

\begin{proposition}\label{prop:properties_of_induced_map}
Retain the above notation. The map $\rho$ is a graph isomorphism from $\Gamma_{++}$ to $\rho(\Gamma_{++})$ where
\[
  V(\rho(\Gamma_{++}))=\{uv_{-i}\mid u\in U_{++}\hbox{, }i \in\bbN\}, \text{ and}
\]
\vspace{-0.4cm}
\[
  E(\rho(\Gamma_{++}))=\{(uv_{-i},uv_{-i-1})\mid u\in U_{++}\hbox{, }i \in\bbN\}.
\]
\end{proposition}

\begin{proof}
We first show $\rho$ is well-defined. Suppose $u_0v_{-i}, u_1v_{-j}\in V(\Gamma_{++})$ represent the same vertex. Lemma \ref{lem:equal_vert_in_gamma} gives $i = j$ and so $u_{0}^{-1}u_1\!\in\! \alpha^{-i}(U)$. Choose $w_0,w_1\in U_{++}$ with $\alpha(w_i) = u_i$ for $i\in\{0,1\}$. Then $\alpha(w_0^{-1}w_{1}) = u_{0}^{-1}u_1\in \alpha^{-i}(U)$ and so $w_0^{-1}w_1\in \alpha^{-i-1}(U)$. This implies $w_0v_{-i-1} = w_{1}v_{-i - 1}$. This shows that setting $\rho(u_0v_{-i}) = w_0v_{-i-1}$ is well-defined. 

To see that $\rho$ is injective suppose that $\rho(u_0v_{-i}) = \rho(u_1v_{-j})$. Then there are $w_0$ and $w_1$ such that $w_0v_{-i - 1}=w_{1}v_{-j - 1}$ and $\alpha(w_i) = u_i$ $(i\in\{0,1\})$. Lemma \ref{lem:equal_vert_in_gamma} gives $i = j$ and so $w_0^{-1}w_1\in \alpha^{-i-1}(U)$. We have
$\alpha(w_0^{-1}w_1) = u_0^{-1}u_1\in \alpha^{-i}(U)$, thus $u_0v_{-1} = u_1v_{-i}$.

As to $V(\rho(\Gamma_{++}))$ we have, $V(\rho(\Gamma_{++}))\supseteq\{uv_{-i}\mid u\in U_{++},\ i\in\bbN\}$ by definition as $\alpha(U_{++})=U_{++}$. Equality follows from Lemma~\ref{lem:equal_vert_in_gamma}.

To see that $\rho$ preserves the edge relation, let $(uv_{-i},uv_{-i - 1})\in E(\Gamma_{++})$. Choose $u'\!\in\! U_{++}$ with $\alpha(u')\!=\! u$. Then $(\rho(uv_{-i}),\rho(uv_{-i-1}))\!=\!(u'v_{-i-1}, u'v_{-i-2})\in E(\Gamma_{++})$. Thus $\rho$ is a graph morphism.

Again, we have $E(\rho(\Gamma_{++}))\supseteq\{(uv_{-i},uv_{-i-1})\mid u\in U_{++}\hbox{, }i \in\bbN\}$ by definition as $\alpha(U_{++})=U_{++}$ and equality by Lemma \ref{lem:equal_vert_in_gamma}.

Finally, to see that $\rho$ is invertible on $\rho(\Gamma_{++})$, we show that $\rho$ maps non-edges to non-edges. Suppose $(u_{0}v_{-n},u_{1}v_{-m})\in V(\Gamma_{++})\times V(\Gamma_{++})$ is not an edge of $\Gamma_{++}$. Then either $m\neq n+1$, or $m=n+1$ but $u_{1}\alpha^{-m}(U)\neq u_{0}\alpha^{-m}(U)$, i.e. $u_{0}^{-1}u_{1}\not\in\alpha^{-m}(U)$. In the first case, $\rho(u_{0}v_{-n},u_{1}v_{-m})$ is not an edge of $\Gamma_{++}$ by Lemma \ref{lem:equal_vert_in_gamma} and the definition of the edge set of $\Gamma_{++}$. In the second case, we have $\rho(u_{0}v_{-n})=u_{0}'v_{-n-1}$ and $\rho(u_{1}v_{-m})=u_{1}'v_{-m-1}$ for some $u_{0}',u_{1}'\in U_{++}$ with $\alpha(u_{0}')=u_{0}$ and $\alpha(u_{1}')=u_{1}$. But $u_{1}'\alpha^{-m-1}\neq u_{0}'\alpha^{-m-1}(U)$ because we have $u_{0}'^{-1}u_{1}'\not\in\alpha^{-m-1}(U)$ given that $u_{0}^{-1}u_{1}\not\in\alpha^{-m}(U)$.
\end{proof}

The following two results capture arc-transitivity of the action of $U_{++}$ on $\Gamma_{++}$.

\begin{lemma}\label{lem:basic_arc_trans}
Retain the above notation. Let $\gamma_0$ and $\gamma_1$ be arcs of equal length in $\Gamma_{++}$ and with origin $uv_0$ ($u\in U_{++})$. Then there is $g\in U_{++}$ such that $g\gamma_{0} = \gamma_{1}$.
\end{lemma}

\begin{proof}
Note that $u^{-1}\gamma_{i}$ $(i\in\{0,1\})$ is an arc with origin $v_0$ and thus is contained in $\desc_{\Gamma_{++}}(v_0)$. Remark \ref{rem:rel_gamma_++_gamma} and Lemma \ref{lem:tidy_above_transitive} show that there exists $u'\in U$ such that $u'u^{-1}\gamma_{0} = u^{-1}\gamma_{1}$. By part \ref{item:thm_tidy_above_U_+_action} of Theorem \ref{thm:tidy_above}, we may choose $u'\in U_{+}$. Then $uu'u^{-1}\in U_{++}$ and $g:=uu'u^{-1}$ serves.
\end{proof}

In the following, we write $[v_0,v_{-k}]$ for the arc $(v_0,\ldots, v_{-k})$. 

\begin{proposition}\label{prop:highly_arc_trans}
Retain the above notation. Let $\gamma_0$ and $\gamma_1$ be arcs in $\Gamma_{++}$ of equal length. Then there are $u\in U_{++}$ and $n\in\bbN_{0}$ with either $u\rho^{n}\gamma_{0} = \gamma_1$ or $u\rho^{n}\gamma_{1} = \gamma_0$. If $\gamma_0$ and $\gamma_1$ both terminate at $v_{-i}$ ($i\in\bbN$), we may choose $n=0$ and $u\in U_{++}\cap U_{--}$.
\end{proposition}

\begin{proof}
Suppose $\gamma_0$ originates at $u_0v_{-i_0}$ and $\gamma_1$ originates at $u_1v_{-i_1}$. Without loss of generality assume $i_0\ge i_1$. Then $\rho^{i_0 - i_1}(\gamma_1)$ originates at $u_1'v_{-i_0}=\rho^{i_0 - i_1}(u_1v_{-i_1})$ for some $u_{1}'\in U_{++}$. For the first assertion it therefore suffices to show that for any two arcs $\gamma_0$ and $\gamma_1$ originating at vertices $u_0v_{-i}$ and $u_1v_{-i}$ ($u_{0},u_{1}\in U_{++}$), there exists $u\in U_{++}$ with $u\gamma_0 = \gamma_1$. Further still, by considering the image of $\gamma_{1}$ under multiplication by $u_0u_1^{-1}$, we can assume the $u_0 = u_1$. Now we can extend $\gamma_j$ to $\gamma_j'$ by concatenating on the left with the path $(u_0v_0,\ldots, u_0v_{-i})$. By Lemma \ref{lem:basic_arc_trans}, there exists $u\in U_{++}$ such that $u\gamma_0' = \gamma_1'$. We must necessarily have $u\gamma_0 = \gamma_1$.

For the second assertion, let $\gamma$ be an arc terminating in $v_{-k}$. It suffices to show that there is $g\in U_{++}\cap U_{--}$ such that $g\gamma\subseteq[v_{0},v_{-k}]$. Extending $\gamma$ if necessary, we can assume without loss of generality that $\gamma$ originates at some $uv_0$ where $u\in U_{++}$.

We now construct $g\in U_{++}\cap U_{--}$ such that $g\gamma = [v_{0}, v_{-k}]$. By Lemma \ref{lem:basic_arc_trans}, there exists $u'\in U_{++}$ such that $u'\gamma = [v_0,v_{-k}]$. Applying Lemma \ref{lem:basic_arc_trans}, for each $n\in\bbN_{0}$ there exist $w_n\in U_{++}$ such that
\[w_n(v_0,\ldots, v_{-k},u'v_{-k - 1},\ldots, u'v_{-k-n}) = [v_{0},v_{-k-n}].\]
The sequence $(w_{n})_{n\in\bbN}$ is contained in $U$ as each element fixes $v_0$. It hence admits a subsequence converging to some $w'\in U$. Put $g:=w'u'\in U_{++}$. Since $\alpha^{-l}(U)$ is closed for all $l\in\bbN_{0}$, we get $g(v_{-l}) = v_{-l}$ for all $l\ge k$. That is, $g\in U_{--}$ and $g\gamma = [v_0,v_{-k}]$.
\end{proof}

\begin{remark}\label{rem:vertex_trans}
Restricting Proposition \ref{prop:highly_arc_trans} to the case where $\gamma_0$ and $\gamma_1$ are single vertices we conclude that for any two vertices $u_0,u_1\in V(\Gamma_{++})$, there are $n\in\bbN_{0}$ and $u\in U_{++}$ such that either $u\rho^{n}(u_0) = u_1$ or $u\rho^{n}(u_1) = u_0$.
\end{remark}

We now show that $\Gamma_{++}$ is locally finite. We will need the following Lemma which is a consequence of \cite[Proposition 4]{Wil15} given that $\calL_{U}$, see \cite[Definition 5]{Wil15}, is precisely $U_{++}\cap U_{--}$.

\begin{lemma}\label{lem:compact_closure}
The closure of $U_{++}\cap U_{--}$ is compact. \qed
\end{lemma}

The last assertion of the following proposition will be used to show that $\Gamma_{++}$ admits a well-defined ``depth'' function.
\begin{proposition}\label{prop:properties_of_gamma_++}
Retain the above notation. The graph $\Gamma_{++}$
\begin{itemize}
 \item[(i)] has constant out-valency,
 \item[(ii)] has constant in-valency among the vertices $\{uv_{-i}\mid u\in U_{++},\ i\in\bbN\}$,
 \item[(iii)] satisfies that the in-valency of $uv_{0}$ ($u\in U_{++}$) is $0$,
 \item[(iv)] is locally finite, and
 \item[(v)] satisfies that every arc from $uv_{-i}$ to $u'v_{-i - k}$ ($u,u'\in U_{++};\ i,k\in\bbN_{0})$ has length $k$.
\end{itemize}
\end{proposition}
\begin{proof}
If $u_0,u_1\in V(\Gamma_{++})$, then by Remark \ref{rem:vertex_trans} and swapping $u_0$ with $u_1$ if necessary, there are $g\in U_{++}$ and $n\in\bbN_{0}$ such that $g\rho^{n}(u_0) = u_1$. Proposition \ref{prop:properties_of_induced_map}, shows that \[|\out_{\Gamma_{++}}(u_1)|=|\out_{\Gamma_{++}}(\rho^{n}(u_0))| = |\out_{\Gamma_{++}}(u_0)|,\] hence (i). Similarly, $|\operatorname{in}_{\Gamma_{++}}(u_0)|=|\operatorname{in}_{\Gamma_{++}}(g\rho^{n}(u_0))|$ if neither $u_0$ and $u_1$ are of the form $uv_0$ for some $u\in U_{++}$ and therefore (ii) holds.

The assertion that $|\operatorname{in}_{\Gamma_{++}}(uv_0)| = 0$ follows since for every edge $(u'v_{-i},u'v_{-i-1})$ we have $u'v_{-i-1}\ne uv_0$ by Lemma \ref{lem:equal_vert_in_gamma}.

For local finiteness it now suffices to show that both $\out_{\Gamma_{++}}(v_0)$ and $\operatorname{in}_{\Gamma_{++}}(v_{-1})$ are finite. Note that by Remark \ref{rem:rel_gamma_++_gamma} we have
\[|\out_{\Gamma_{++}}(v_0)|=|Uv_{-1}| = [U:U\cap \alpha^{-1}(U)]\]
which is finite by compactness of $U$ and continuity of $\alpha$. To see that $\operatorname{in}_{\Gamma_{++}}(v_{-1})$ is finite, note that by Proposition \ref{prop:highly_arc_trans} each vertex of $\operatorname{in}_{\Gamma_{++}}(v_{-1})$ can be written as $uv_{0}$ where $u\in U_{++}\cap U_{--}\cap \alpha^{-1}(U)$. Conversely, any such $u$ yields a vertex in $\mathrm{in}_{\Gamma_{++}}(v_{-1})$. Thus
\[|\operatorname{in}_{\Gamma_{++}}(v_{-1})| = [U_{++}\cap U_{--}\cap \alpha^{-1}(U):U_{++}\cap U_{--}\cap \alpha^{-1}(U)\cap U].\]
If $u_{0},u_1\in U_{++}\cap U_{--}\cap \alpha^{-1}(U)$ with $u_0u_1^{-1}\not\in U$, then  $u_{0},u_{1}\!\in\!\overline{U_{++}\cap U_{--}}\cap \alpha^{-1}(U)$ a fortiori and $u_0u_1^{-1}\not\in U$. Thus
\begin{displaymath}
|\operatorname{in}_{\Gamma_{++}}(v_{-1})|\le [\overline{U_{++}\cap U_{--}}\cap \alpha^{-1}(U): \overline{U_{++}\cap U_{--}}\cap \alpha^{-1}(U)\cap U].
\end{displaymath}
Applying Lemma \ref{lem:compact_closure} and noting that $\alpha^{-1}(U)$ is closed, $\overline{U_{++}\cap U_{--}}\cap \alpha^{-1}(U)$ is compact. Furthermore, since $U$ is open, we derive that $\overline{U_{++}\cap U_{--}}\cap \alpha^{-1}(U)\cap U$ is open in $\overline{U_{++}\cap U_{--}}\cap \alpha^{-1}(U)$. Thus $\mathrm{in}_{\Gamma_{++}}(v_{-1})$ is finite.

For part (v), let $(w_{0}, ..., w_{n})$ be an arc from $uv_{−i}$ to $u'v_{−i−k}$ . By definition, we have $w_{1} = u_{n}v_{-i-1}$ for some $u_{1}\in U_{++}$. Hence, by induction, we get $w_{n} = u_{n}v_{-i-n}$ for some $u_{n}\in U_{++}$. But $w_{n}=u'v_{-i-k}$ , so that Lemma 4.2 implies $n = k$.
\end{proof}

\subsection{The quotient $T$}
The tidying procedure relies on identifying a certain quotient $T$ of $\Gamma_{++}$ as a forest of regular rooted trees. To define this quotient, we first introduce a ``depth'' function $\psi:V(\Gamma_{++})\to \bbN$ on $\Gamma_{++}$ as follows: For $v\in V(\Gamma_{++})$, choose an arc $\gamma$ originating from some $uv_0$ ($u\in U_{++}$) and terminating at $v$. Set $\psi(v)$ to be the length of $\gamma$. The following is immediate from Proposition \ref{prop:properties_of_gamma_++}.

\begin{lemma}\label{lem:level_sets}
Retain the above notation. The map $\psi$ is well-defined and we have $\psi(uv_{-i}) = i$ for all $u\in U_{++}$ and $i\in\bbN_{0}$. \qed
\end{lemma}

By virtue of Lemma \ref{lem:level_sets} we may define the level sets $V_{k}:=\psi^{-1}(k)\subseteq V(\Gamma_{++})$ for $k\ge 0$ and
the edge sets $E_{k}:=\{(w,w')\in E(\Gamma_{++})\mid\psi(w')=k\}$ for $k\ge 1$. It is a consequence of Lemma \ref{lem:level_sets} and Lemma \ref{lem:equal_vert_in_gamma} that  $(w,w')\in E_{k}$ if and only if there is $u\in U_{++}$ such that $(w,w')=(uv_{-k + 1}, uv_{-k})$. On $V_{k}$ $(k\ge 1)$ we introduce an equivalence relation by $w\sim w'$ if $w$ and $w'$ belong to the same connected component of $\Gamma_{++}\backslash E_{k}$. Similarly, for $w,w'\in V_{0}$ we put $w\sim w'$ if they belong to the same connected component of $\Gamma_{++}$. Write $[w]$ for the collection of vertices $w'$ with $w\sim w'$. Note that for every $g\in U_{++}$ and $k\in\bbN_{0}$ we have $gV_{k} = V_{k}$ and, if $k\ge 1$, $gE_{k} = E_{k}$. Since the action of $U_{++}$ on $\Gamma_{++}$ preserves connected components we see that $w\sim w'$ if and only if $gw\sim gw'$. The following Lemma extends this to $\rho$.

\begin{lemma}\label{lem:equiv_classes_preserved_by_action}
Retain the above notation and let $k\in\bbN_{0}$. Then $\rho(V_{k})\!=\! V_{k+1}$ and, given $k\ge 1$, $\rho(E_{k}) = E_{k+1}$. Hence, for $w,w'\!\in\! V(\Gamma_{++})$ we have $w\!\sim\! w'$ if and only if $\rho(w)\!\sim\! \rho(w')$.
\end{lemma}
 
\begin{proof}
The assertions $\rho(V_k) = V_{k+1}$ and $\rho(E_k) = E_{k+1}$ are immediate from the definitions. Suppose now that $k\ge 1$ and $w,w'\in V_{k}$ are in the same connected component of $\Gamma_{++}\setminus E_{k}$. By Proposition \ref{prop:properties_of_induced_map}, this can occur if and only if $\rho(w),\rho(w')\in V_{k+1}$ are in the same connected component of $\rho(\Gamma_{++})\setminus E_{k+1}$. The embedding of $\rho(\Gamma_{++})$ into $\Gamma_{++}$ maps connected components of $\rho(\Gamma_{++})\!\setminus\! E_{k+1}$ to connected components of $\Gamma_{++}\!\setminus E_{k+1}$ and is surjective on $V_{k+1}$.

For the case $k = 0$ suppose $w,w'\in V_0$ are in the same connected component of $\Gamma_{++}$. We then have $\rho(w)$ and $\rho(w')$ in the same connected component of $\rho(\Gamma_{++})$. Since $\rho(\Gamma_{++})$ is disjoint from $E_1$ by Proposition \ref{prop:properties_of_induced_map} and Lemma \ref{lem:level_sets}, we in fact conclude that $\rho(w)$ and $\rho(w')$ are in the same connected component of $\rho(\Gamma_{++}) = \rho(\Gamma_{++})\setminus E_1\subset \Gamma_{++}\setminus E_1$.
\end{proof}

\begin{lemma}\label{lem:in_vertices_in_descendants_of_root}
Retain the above notation. There is $N\in\bbN$ such that for every $v\in\desc_{\Gamma_{++}}(v_{0})$ with $\psi(v)\ge N$ we have $\operatorname{in}_{\Gamma_{++}}(v)\subseteq \desc_{\Gamma_{++}}(v_0)$. 
\end{lemma}

\begin{proof}
By Proposition \ref{prop:properties_of_gamma_++}, we can choose $u_0,\ldots, u_k\in U_{++}\cap \alpha^{-1}(U)$ such that $\operatorname{in}_{\Gamma_{++}}(v_{-1}) = \{u_0v_0,\ldots, u_{k}v_0\}$. Since $u_{i}\in U_{++}$ for all $i\in\{0,\ldots, k\}$, we may pick $\alpha$-regressive trajectories $(w_{j}^{i})_{j\in\bbN_{0}}$ and $N_i\in\bbN$ such that $w_{0}^{i} = u_{i}$ and $w^{i}_{n}\in U$ for all $n\ge N_{i}$. Set $N = \max\{N_i\mid i\in\{0,\ldots,k\}\}+1$.

Suppose $n\ge N$. To see that $\operatorname{in}_{\Gamma_{++}}(v_{-n})\subseteq\desc_{\Gamma_{++}}(v_0)$ note that by Proposition \ref{prop:properties_of_gamma_++} we have $\operatorname{in}_{\Gamma_{++}}(v_{-n}) = \rho^{n - 1}(\operatorname{in}(v_{-1})) = \{w_{n-1}^{i}v_{-N +1}\mid  i\in\{0,\ldots,k\}\}$. Since $n - 1\ge N_i$ for all $i\!\in\!\{0,\ldots , k\}$, the path $(w_{n-1}^{i}v_0,\ldots, w_{n-1}^{i}v_{-n+1})$ is contained in $\desc_{\Gamma_{++}}(v_{0})$. This shows $\operatorname{in}_{\Gamma_{++}}(v_{-n})\subseteq \desc_{\Gamma_{++}}(v_0)$. 

In general, let $v\in \desc_{\Gamma_{++}}(v_0)$ with $\psi(v)=n\ge N$. Applying Proposition \ref{prop:highly_arc_trans} to the arc $(v_0,\ldots, v_{-n})$ and any arc connecting $v_0$ to $v$, there is $u\in U\cap U_{++}$ such that $uv_{-n} = v$. Furthermore, $u\desc_{\Gamma_{++}}(v_0) = \desc_{\Gamma_{++}}(v_0)$ as $uv_0 = v_0$ and it follows that $\operatorname{in}_{\Gamma_{++}}(v) = u\operatorname{in}_{\Gamma_{++}}(v_{-n})\subseteq \desc_{\Gamma_{++}}(v_0)$.   
\end{proof}

\begin{lemma}\label{lem:equiv_classes_are_constant}
Retain the above notation. Then the equivalence classes on $\Gamma_{++}$ induced by $\sim$ have finite constant size.
\end{lemma}

\begin{proof}
By Lemma \ref{lem:equiv_classes_preserved_by_action} and the fact that $U_{++}$ acts transitively on $V_{k}$ for every $k\in\bbN_{0}$, it suffices to show that a single equivalence is finite. Using Lemma \ref{lem:in_vertices_in_descendants_of_root}, choose $N\in\bbN$ such that for every $v\in \desc_{\Gamma_{++}}(v_{0})$ with $\psi(v)\ge N$ we have $\operatorname{in}_{\Gamma_{++}}(v)\subset \desc_{\Gamma_{++}}(v_0)$. We show that $[v_{-N}]\!\subseteq\! \desc_{\Gamma_{++}}(v_0)$. Since $\desc_{\Gamma_{++}}(v_0)\cap V_{k}$ is finite for all $k\in\bbN$ by Proposition \ref{prop:properties_of_gamma_++} and Lemma \ref{lem:equal_vert_in_gamma}, the assertion follows.

Suppose $v\in [v_{-N}]$. Then $v_{-N}$ and $v$ are in the same connected component of $\Gamma_{++}\!\setminus E_{N}$. Hence there is a path from $v_{-N}$ to $v$ contained in $\Gamma_{++}\setminus E_{N}$. Choosing arcs within this path and extending them to $V_{N}$ if necessary, we see that there are vertices $u_0,\ldots, u_{n}\in V_{N}$ with $u_0 = v_{-N}$, $u_{n} = v$ and $\desc_{\Gamma_{++}}(u_{i})\cap \desc_{\Gamma_{++}}(u_{i+1})\neq \emptyset$. We use induction to show that $u_{i}\in \desc_{\Gamma_{++}}(v_0)$. Clearly, $u_0 = v_{-N}\in \desc_{\Gamma_{++}}(v_0)$. Suppose $u_k\in \desc_{\Gamma_{++}}(v_0)$ and let $(w_{0},\ldots,w_{l})$ be an arc such that $w_{0} = u_{k+1}$ and $w_l\!\in\! \desc_{\Gamma_{++}}(u_k)\cap \desc_{\Gamma_{++}}(u_{k+1})$. Then $w_{l}\!\in\! \desc_{\Gamma_{++}}(v_0)$ and $\psi(w_{-l}) = N + l > N$. This implies $w_{l - 1}\in \operatorname{in}_{\Gamma_{++}}(w_l)\subseteq \desc_{\Gamma_{++}}(v_0)$ by the choice of $N$. Repeating this process until $u_{k+1}\!=\! w_{0}\!\in\! \operatorname{in}_{\Gamma_{++}}(w_{1})\!\subseteq\! \desc_{\Gamma_{++}}(v_0)$ completes the induction.
\end{proof}

Now define a directed graph $T$ as the quotient of $\Gamma_{++}$ by the vertex equivalence relation introduced above. In particular, $([w],[w'])$ is an edge in $T$ if and only if there are representatives $w\in[w]$ and $w'\in[w']$ such that $(w,w')$ is an edge in $\Gamma_{++}$. The following result collects properties of $T$. For the statement, we let $d_{+} = |\out_{\Gamma_{++}}(v_0)|$ and $d_{-} = |\operatorname{in}_{\Gamma_{++}}(v_{-1})|$. We let $\varphi:\Gamma_{++}\to T$ denote the quotient map.

\begin{lemma}\label{lem:properties_of_T}
Retain the above notation. The quotient $T$ is a forest of regular rooted trees of degree $d_+/d_-$. The map $\rho$ and the action of $U_{++}$ on $\Gamma_{++}$ descend to $T$. Furthermore, we have the following.
\begin{itemize}
  \item[(i)] The map $\rho$ is a graph morphism from $T$ onto $\rho(T)$ where
  \[V(\rho(T)) = \{[uv_{-i}]\mid u\in U_{++},\ i \in \bbN \}, \text{ and}\]
  \vspace{-0.4cm}  
  \[E(\rho(T)) = \{([uv_{-i}], [uv_{-i - 1}])\mid u\in U_{++},\ i \in \bbN\}.\]
  \item[(ii)] For every $v\in V(T)$, the stabilizer $(U_{++})_v$ acts transitively on $\out_{T}(v)$.
\end{itemize}
\end{lemma}

\begin{proof}
It is clear that if $v\in V(\Gamma_{++})\cap V_{0}$, then $|\operatorname{in}_{T}([v])| = 0$ since $|\operatorname{in}_{\Gamma_{++}}(u)| = 0$ for all $u\in V_{0}$. We now show that if $v\in \Gamma_{++}\!\setminus V_{0}$, then $|\operatorname{in}_{T}([v])| = 1$. Since $|\operatorname{in}_{\Gamma_{++}}(v)|\ge 1$, we have $|\operatorname{in}_{T}([v])|\ge 1$. Suppose now that $([u_0], [v])$ and $([u_1],[v])$ are edges in $T$. Then there are representatives $u_i', w_i'\in V(\Gamma_{++})$ such that $u_i'\in [u_i]$, $w_i\in [v]$ and $(u_i',w_i')\in E(\Gamma_{++})$ for $i\in\{0,1\}$. In particular, $w_0$ is in the same connected component of $\Gamma_{++}\!\setminus E_{\psi(w_0)}$ as $w_1$. Consequently, $u_0'$ is in the same connected component of $E_{\psi(w_0)-1}$ as $u_1'$. As $\psi(u_0') \!=\! \psi(w_0) - 1 \!=\! \psi(w_1)-1 \!=\!\psi(u_1')$, this shows that $[u_0] = [u_0'] = [u_1'] = [u_1]$, and so $([u_0], [v]) = ([u_1], [v])$. We therefore have $|\operatorname{in}_T([v])| = 1$.

The map $\rho$ descends to $T$ by Lemma \ref{lem:equiv_classes_preserved_by_action}, and the action of $U_{++}$ on $\Gamma_{++}$ readily descends to $T$. The assertions concerning $\rho$ and $\rho(T)$ are immediate from Proposition \ref{prop:properties_of_induced_map}. It follows from Proposition \ref{prop:properties_of_induced_map}, that $u\rho^{n}(\operatorname{in}_{T}(v)) = \operatorname{in}_{T}(u\rho^{n}(v))$ for all $v\in V(T)\backslash\varphi(V_{0})$. Since $\rho$ and the action of $U_{++}$ descend to $T$, an analogue of Remark \ref{rem:vertex_trans} also holds for $T$: Indeed, let $[u_{0}],[u_{1}]\in V(T)$. By Remark \ref{rem:vertex_trans}, there is $u\in U_{++}$ and $n\in\bbN_{0}$ such that, without loss of generality, $u\rho^{n}(u_{0})=u_{1}$. Hence $u\rho^{n}([u_{0}])=[u\rho^{n}(u_{0})]=[u_{1}]$. Therefore, $T$ is a forest of regular rooted trees and has constant out-valency.

Let $d$ denote the out-valency of $T$. As in \cite[Lemma 5]{Moe00}, we argue that $d=d^{+}/d^{-}$. By Lemma \ref{lem:equiv_classes_are_constant}, equivalence classes of vertices in $\Gamma_{++}$ have constant finite order $k\in\bbN$. Given $v\in V(T)$, let $A:=\varphi^{-1}(v)$. The $d$ edges issuing from $v$ end in vertices $w_{1},\ldots,w_{d}\in V(T)$. Put $B:=\varphi^{-1}(\{w_{1},\ldots,w_{d}\})$. Then all edges in $\Gamma_{++}$ ending in $B$ originate in $A$ because $T$ has in-valency $1$. The number of edges issuing from $A$, which is $kd^{+}$, and the number of edges terminating in $B$, which is $kdd^{-}$, are thus equal. Hence $d=d^{+}/d^{-}$.

For (ii), let $v\in V(T)$ and $u_0,u_1\in\out_{T}(v)$. Pick representatives $w_0,w_0', w_1,w_1'$ in $V(\Gamma_{++})$ such that $([w_i],[w_i']) = (v,u_i)$ for $i\in\{0,1\}$ and choose $g\in U_{++}$ such that $g(w_0,w_0') = (w_1,w_1')$ by Proposition \ref{prop:highly_arc_trans}. Then $gv = v$ and $gu_0 = u_1$.
\end{proof}

\begin{theorem}\label{thm:tidy_existence}
Let $G$ be a t.d.l.c. group and $\alpha\in\End(G)$. Then there exists a compact open subgroup $V\le G$ which is tidy for $\alpha$.
\end{theorem}

\begin{proof}
Recall that throughout the section we consider a compact open subgroup $U\le G$ which, referring to Lemma \ref{lem:finite_vertices_case}, we assume to satisfy that $\{\alpha^{-i}(U)\mid i\in\bbN_{0}\}$ is infinite, and, referring to Theorem \ref{thm:tidy_above_exists}, we assume to be tidy above.

For $i\in\bbN_{0}$, let $v_{-i}':=\varphi(v_{-i})\in V(T)$. In view of the fact that $\Gamma_{++}\subseteq\Gamma$, consider $V:=G_{\{X_{0}\}}$ where $X_{0}:=[v_{0}]\subseteq V(\Gamma_{++})$ is the equivalence class of $v_{0}$ in $\Gamma_{++}$. Since $X_{0}$ is finite by Lemma \ref{lem:equiv_classes_are_constant}, $V$ is open in the permutation topology coming from $\Gamma$ given that $G_{X_{0}}\le V=G_{\{X_{0}\}}=\{g\in G\mid gX_{0}=X_{0}\}$, and hence also open (and closed) in $G$. For the same reason we conclude that $V$ is compact as it contains the compact group $U$ as a finite index subgroup.

Recall that by Remark \ref{rem:rel_gamma_++_gamma}, $\desc_{\Gamma_{++}}(v_0) = \desc_{\Gamma}(v_0)$. Since $U_{++}$ acts by automorphisms on $\Gamma_{++}$ and is transitive on $X_0$, we have $\desc_{\Gamma_{++}}(v) = \desc_{\Gamma}(v)$ for all $v\in X_0$. Taking the union of vertices $v\in X_0$ we have $\desc_{\Gamma_{++}}(X_0) = \desc_{\Gamma}(X_0)$. As $V\! \desc_{\Gamma}(X_0) = \desc_{\Gamma}(X_0)$, the group $V$ acts on $\desc_{\Gamma_{++}}(X_0)$ by automorphisms.

It is clear that $V$ preserves $V_{k}$, $E_{k}$ and connected components. So the action of $V$ descends to $T$ and $V$ stabilizes $v_0'\in V(T)$. Note that $(U_{++})_{v_0'}\le V$ and so iterated application of Lemma \ref{lem:properties_of_T} shows that $V$ acts transitively on vertices of fixed depth in $T$. Also, $V_{v_{-i}'} = V\cap \alpha^{-i}(V)$: Suppose $g\in V$ and $gv_{-i} = uv_{-i}$, where $u\in U_{++}$. Then $g^{-1}u\in \alpha^{-i}(U)$. Thus $\smash{\alpha^{i}(g^{-1}u)\in U}$ and so $\smash{\alpha^{i}(g)v_0 = \alpha^{i}(u)v_0}$. Applying Lemma \ref{lem:equiv_classes_preserved_by_action}, we see that $gv_i\sim v_i$ if and only if $\alpha^{i}(g)v_0\sim v_0$. Finally, applying the orbit-stabilizer theorem and Lemma \ref{lem:properties_of_T} we have
\[[V:V\cap \alpha^{-n}(V)] = |Vv_{-n}'| = (d_{+}/d_{-})^{n} = |Vv_{-1}'|^{n} = [V:V\cap \alpha^{-1}(V)]^{n}.\]
for all $n\in\bbN$. Hence $V$ is tidy for $\alpha$ by Lemma~\ref{lem:tidy_powers}
\end{proof}

\begin{remark}\label{rem:gamma_++_tidy}
Retain the above notation and assume that $U$ is tidy. We argue that in this case $\Gamma_{++}$ and $T$ coincide: It suffices to show that $|\mathrm{in}_{\Gamma_{++}}(v)| = 1$ for some $v = uv_{-i}$ with $i > 0$ as Proposition \ref{prop:properties_of_gamma_++} shows that the relation $\sim$ on $\Gamma_{++}$ is trivial. By Remark \ref{rem:rel_gamma_++_gamma} and Theorem \ref{thm:tidy_tree}, the graph $\desc_{\Gamma_{++}}(v_{0})=\Gamma_{+}$ is already a tree. Lemma \ref{lem:in_vertices_in_descendants_of_root} shows that there exists a vertex $v$ with $\operatorname{in}_{\Gamma_{++}}(v)\subset \Gamma_{+}$. Thus $|\mathrm{in}_{\Gamma_{++}}(v)| = 1$.
\end{remark}

The following lemma will be used in Section \ref{sec:tree_rep_thm}.

\begin{lemma}\label{lem:pseudo_nub}
Suppose $U$ is tidy for $\alpha$. Then $U_{++}\cap U_{--}\le U_+\cap U_-\le U$.
\end{lemma}

\begin{proof}
Since $U$ is tidy for $\alpha$, the graph $\Gamma_{++}$ is a forest of rooted trees by Remark \ref{rem:gamma_++_tidy} . Note that for each $u\in U_{++}\cap U_{--}$, there exists $i\in \bbN_{0}$ such that $uv_{-i} = v_{-i}$. Hence $U_{++}\cap U_{--}$ preserves $\desc_{\Gamma_{++}}(v_0)$, which is a tree with root $v_0$, and is therefore contained in $\stab_{G}(v_{0})=U$. The claim now follows from Lemma~\ref{lem:tidy_below}.
\end{proof}

\section{The Scale Function and Tidy Subgroups}\label{sec:scale_function}

In this section we link the concept of tidy subgroups to the scale function and thereby recover results of \cite{Wil15} in a geometric manner. First, we make a preliminary investigation into the intersection of tidy subgroups. Let $G$ be a t.d.l.c. group, $\alpha\in\End(G)$ and $U^{(1)},U^{(2)}\le G$ compact open as well as tidy for $\alpha$.

\begin{proposition}\label{prop:tidy_subgroups_index_equal}
Retain the above notation. Then
\begin{displaymath}
 [U^{(1)}:U^{(1)}\cap \alpha^{-1}(U^{(1)})] = [U^{(2)}:U^{(2)}\cap \alpha^{-1}(U^{(2)})]
\end{displaymath}
\end{proposition}

To prove Proposition \ref{prop:tidy_subgroups_index_equal}, we need some preparatory lemmas concerning inverse images of $U^{(1)}$ and $U^{(2)}$. The first one complements Lemma \ref{lem:finite_vertices_case}.

\begin{lemma}\label{lem:finite_tidy_stabilize}
Let $G$ be a t.d.l.c. group, $\alpha\in\End(G)$ and $U\le G$ compact open and tidy above for $\alpha$. If $\{\alpha^{-n}(U)\mid n\in\bbN_{0}\}$ is finite, then $\alpha(U)\le U\le\alpha^{-1}(U)$. \end{lemma}

\begin{proof}
By assumption, the intersection $\bigcap_{k=0}^{\infty}\alpha^{-k}(U)$ has only finitely many terms and hence stabilizes eventually. For sufficiently large $n\in\bbN_0$ we therefore have $[U_{-n}: U_{-n-1}] = 1$. By Proposition \ref{prop:out_index}, we get
\begin{displaymath}
  1\!=\! [U_{-n}:U_{-n-1}]\!=\! [U:U_{-1}]=[U:U\cap\alpha^{-1}(U)].
\end{displaymath}
Hence $U\le\alpha^{-1}(U)$. Applying $\alpha$ yields $\alpha(U)\le U$.
\end{proof}

The next lemma settles Proposition \ref{prop:tidy_subgroups_index_equal} when both $\{\alpha^{-n}(U^{(1)})\mid n\in\bbN_{0}\}$ and $\{\alpha^{-n}(U^{(2)})\mid n\in\bbN_{0}\}$ are finite.

\begin{lemma}\label{lem:both_finite_case}
Retain the above notation. If $\{\alpha^{-n}(U^{(i)})\mid n\in\bbN_{0}\}$ is finite for both $i\in\{1,2\}$, then $[U^{(1)}: U^{(1)}\cap\alpha^{-1}(U^{(1)})]\!=1=\![U^{(2)}:U^{(2)}\cap \alpha^{-1}(U^{(2)})]$ and the group $U^{(1)}\cap U^{(2)}$ is tidy for $\alpha$. 
\end{lemma}

\begin{proof}
The first assertion follows from Lemma \ref{lem:finite_tidy_stabilize} as both indices are equal to $1$. By the same Lemma we have $\smash{\alpha^{-1}(U^{(1)}\cap U^{(2)})=\alpha^{-1}(U^{(1)})\!\cap\!\alpha^{-1}(U^{(2)})\ge U^{(1)}\cap U^{(2)}}$. Lemma \ref{lem:finite_vertices_case} now entails that $\smash{(U^{(1)}\cap U^{(2)})_{-} = U^{(1)}\cap U^{(2)}}$ is tidy for $\alpha$.
\end{proof}

Retain the above notation and set $V:=U^{(1)}\cap U^{(2)}$. Consider the graph $\Gamma_{+}$ associated to $V$.

\begin{lemma}\label{lem:tree_or_point}
Retain the above notation. Then either $\Gamma_{+}$ is a directed infinite tree, rooted at $v_{0}$, with constant in-valency $1$ excluding the root, or there exists $n\in\bbN_{0}$ such that $\alpha^{-n}(V)\!=\!\alpha^{-n - k}(V)$ for all $k\in\bbN_{0}$.
\end{lemma}

\begin{proof}
Note that if $\alpha^{-n}(V) = \alpha^{-n - 1}(V)$, then $\alpha^{-n}(V) = \alpha^{- n - k}(V)$ for all $k\in\bbN_{0}$. Suppose instead that $\alpha^{-n}(V)\neq\alpha^{-n-1}(V)$ for all $n\in\bbN_{0}$.Then we may assume, without loss of generality, that $\{\alpha^{-n}(U^{(1)})\mid n\in\bbN_{0}\}$ is infinite. In particular, we may consider the graph $\smash{\Gamma_{+}^{(1)}}$ associated to $\smash{U^{(1)}}$ which is an infinite rooted tree by Theorem \ref{thm:tidy_tree}.

We have to show that $\Gamma_{+}$ does not contain a cycle, the in-valency of $v_0\in V(\Gamma_{+})$ is $0$ and the in-valency of every other vertex in $\Gamma_{+}$ is precisely $1$. Note that every vertex excluding $v_0$ has in-valency at least $1$: By assumption, $v_{-i}\neq v_{-i-1}$ for all $i\in\bbN$. In particular $v_{-i}\in\operatorname{in}_{\Gamma_{+}}(v_{-i-1})$ for all $i\in\bbN$. 

Now, suppose that $c:=(w_{0},\ldots,w_{n}=w_{0})$ $(n\ge 2)$ is a cycle in $\Gamma_{+}$. In particular, the vertices $w_{0},\ldots,w_{n}\in V(\Gamma_{+})$ are pairwise distinct. We show that $c$ forces some vertex of $\Gamma_{+}$ to have in-valency at least $2$ and then argue that such vertices do not exist. If $c$ is not directed, then there is $i\in\{0,\ldots,n\}$ such that the in-valency of $w_{i}$ is at least $2$. Otherwise, $c$ is directed of the form $(u_0v_{-i},\ldots, u_{n}v_{-i-n} = u_0v_{-i})$, where $u_{j}\in V$ for all $j\in\{0,\ldots,n\}$. Then $\alpha^{-i}(V) = \alpha^{-i-n}(V)$ and therefore $(v_{-i},\ldots, v_{-i - n})$ is a non-trivial cycle. We aim to show that $v_{-i}$ has in-valency at least 2 in this case. We can choose $u\!\in\!\alpha^{-i-1}(V)\!\setminus\!\alpha^{-i}(V)$: If $\alpha^{-i-1}(V)\!\subseteq\!\alpha^{-i}(V)$, then iterated applications of $\alpha^{-1}$ show $\alpha^{-i}(V)\!\supseteq\! \alpha^{-i-1}(V)\!\supseteq\!\alpha^{-i-n}(V)\!=\!\alpha^{-i}(V)$, in contradiction to the assumption. Since $\alpha^{-i - 1}(V)\!=\!\alpha^{-1}\alpha^{-i}(V)\!=\!\alpha^{-1}\alpha^{-i - n}(V)$, also $u\in\alpha^{-i-n - 1}(V)\!\setminus\!\alpha^{-n-i}(V)$. This implies that $(uv_{-i-n},v_{-i -n -1})$ is an edge in $\Gamma_{+}$ which is distinct from $(v_{-n-i},v_{-n - i - 1})$.

Noting that if $v_0$ has non-zero in-valency, then we have a cycle, it remains to show that no vertex has in-valency at least $2$. We split into two cases: First, consider the case where $\smash{\{\alpha^{-n}(U^{(2)})\!\mid\! n\in\bbN_{0}\}}$ is finite. Then $\smash{\alpha^{-n}(U^{(2)})\!\ge\!\alpha^{-n+1}(U^{(2)})\!\ge\! U^{(2)}}$ for all $n\in\bbN$ by Lemma \ref{lem:finite_tidy_stabilize} and
\begin{align*}
|\operatorname{in}&_{\Gamma_{+}}(v_i)| = [\alpha^{-i}(V)\cap V: \alpha^{-i}(V)\cap V\cap\alpha^{-i+1}(V)] \\
&=[\alpha^{-i}(U^{(1)})\cap U_{1}\cap U_{2}:\alpha^{-i}(U^{(1)})\cap U_{1}\cap\alpha^{-i + 1}(U^{(1)})\cap U_{2}] \\
&\le [\alpha^{-i}(U^{(1)})\cap U_{1}:\alpha^{-i}(U^{(1)})\cap U_{1}\cap\alpha^{-i + 1}(U^{(1)})] = |\operatorname{in}_{\Gamma_{+}^{(1)}}(v_{-i}^{(1)})| = 1
\end{align*}
for all $i\in\bbN$ which suffices.

In the case where $\{\alpha^{-n}(U^{(2)})\mid n\in\bbN_{0}\}$ is infinite, suppose for the sake of a contradiction that $uv_{-n}\in V(\Gamma_{+})$ $(n\in\bbN)$ has in-valency at least $2$. Choose vertices $wv_{-n+1},zv_{-n+1}\in V(\Gamma_{+})$ such that $(wv_{-n+1},uv_{-n})$ and $(zv_{-n+1},vv_{-n})$ are distinct edges in $\Gamma_{+}$. Let $\smash{\varphi_{i}:\Gamma_{+}\to \Gamma_{+}^{(i)}}$ ($i\in\{1,2\}$) be the graph morphism given by $\smash{\varphi_{i}(uv_{-j}) = uv_{-j}^{(i)}}$ for all $j\in\bbN_{0}$ and $\smash{u\in V\subseteq U^{(i)}}$. Since each vertex excluding the root in $\smash{\Gamma_{+}^{(i)}}$ has in-valency $1$, we have $\varphi_{i}(wv_{-n+1}) = \varphi_{i}(zv_{-n+1})$. This implies $w^{-1}z\in \alpha^{-n+1}(U^{(1)})\cap \alpha^{-n+1}(U^{(2)}) = \alpha^{-n+1}(V)$. We therefore have $wv_{-n+1} = zv_{-n+1}$, in contradiction to the assumption.
\end{proof}

Set $k_i = [U^{(i)}:V]$ and $d_i = [U^{(i)}:U^{(i)}\cap \alpha^{-1}(U^{(i)})]$.
\begin{lemma}\label{lem:orbit_bounds}
Retain the above notation. We have $k_id_i^{n}\ge |Vv_{-n}|\ge d_i^n/k_i$. Also, if $\{\alpha^{-i}(V)\mid i\in\bbN_{0}\}$ is finite, then $d_1=1=d_{2}$.
\end{lemma}

\begin{proof}

Since $U^{(i)}$ is tidy, either the graph $\Gamma_{+}^{(i)}$ is a tree with out-valency $d_i$ by Theorem \ref{thm:tidy_tree}, or $\smash{\{\alpha^{-i}(U^{(i)})\mid i\in\bbN_{0}\}}$ is finite. Then $d_{1}=1=d_{2}$ by Lemma~\ref{lem:both_finite_case}. In both cases, $\smash{k_id_i^{n} = k_i|U^{(i)}v_{-n}^{(i)}|}$, as the following arguments show: In the former case this follows from Lemma \ref{lem:tidy_above_transitive}, in the latter we have $\smash{v_{-n}^{(i)} = v_0^{(i)}}$ whence $\smash{|U^{(i)}v_{-n}^{(i)}|=1}$. Next, we have
\[k_i|U^{(i)}v_{-n}^{(i)}| = [U^{(i)}:V][U^{({i})}:U^{(i)}\cap \alpha^{-n}(U^{(i)})].\]
Since $[\alpha^{-n}(U^{(i)}):\alpha^{-n}(V)]\le [U^{(i)}:V]$ we obtain
\begin{align*}
k_i|U^{(i)}v_{-n}^{(i)}|&\ge[U^{({i})}:U^{(i)}\cap \alpha^{-n}(U^{(i)})][\alpha^{-n}(U^{(i)}):\alpha^{-n}(V)] \\
&\ge [U^{({i})}:U^{(i)}\cap \alpha^{-n}(U^{(i)})][\alpha^{-n}(U^{(i)})\cap U^{(i)}:U^{(i)}\cap \alpha^{-n}(V)] \\
&=[U^{(i)}:U^{(i)}\cap\alpha^{-n}(V)] \\
&= |U^{(i)}v_{-n}|
\end{align*}
where $U^{(i)}v_{-n}$ is the orbit of $v_{-n}$ under the action $U^{(i)}$ in $\calP(G)$. Since $V\le U^{(i)}$, we have $k_id_i^{n}\ge|U^{(i)}v_{-n}|\ge|Vv_{-n}|$ which is the first inequality.

Since $\smash{\alpha^{-n}(V) = \alpha^{-n}(U^{(1)})\cap \alpha^{-n}(U^{(2)})\le\alpha^{-n}(U^{(i)})}$, we have $\smash{|Vv_{-n}|\ge |Vv_{-n}^{(i)}|}$ when considered as orbits in $\calP(G)$. The orbit-stabilizer theorem now implies
\begin{align*}
|Vv_{-n}^{(i)}| &= \dfrac{[U^{(i)}:V][V: \stab_{V}(v_{-n}^{(i)})]}{[U^{(i)}:V]} =\dfrac{[U^{(i)}:\stab_{V}(v_{-n}^{(i)})]}{k_i}\\
& \ge \dfrac{[U^{(i)}:\stab_{U^{(i)}}(v_{-n}^{(i)})]}{k_i} = \dfrac{|U^{(i)}v_{-n}^{(i)}|}{k_i} = \dfrac{d_i^{n}}{k_i},
\end{align*}
as required. Finally, if $\{\alpha^{-i}(V)\mid i\in\bbN_{0}\}$ is finite, then $\alpha^{-n}(V) = \alpha^{-n - k}(V)$ for $n$ sufficiently large and $k\in\bbN_{0}$ by Lemma \ref{lem:tree_or_point}. Thus $(|Vv_{-n}|)_{n\in\bbN_{0}}$ eventually stabilizes. This implies $d_{i}=1$.
\end{proof}

\begin{proof}[Proof of Proposition \ref{prop:tidy_subgroups_index_equal}] By Lemma \ref{lem:orbit_bounds}, we may assume that $\{\alpha^{-i}(V)\mid i\in\bbN_{0}\}$ is infinite. In this case, Lemma \ref{lem:tree_or_point} shows that $\Gamma_{+}$ is a rooted tree with root $v_0$. Let $t_n = |\out_{\Gamma_{+}}(v_{-n})|$ for $n\in\bbN_{0}$. Since $\Gamma_{+}$ is a rooted tree, $t_n = [V_{-n}:V_{-n-1}]$. 

The sequence $(t_{n})_{n\in\bbN_{0}}$ is non-increasing: Indeed, we have
\begin{displaymath}
  t_{n-1} = [V_{-n+1}:V_{-n}]\ge [V_{-n}:V_{-n-1}] = t_{n}
\end{displaymath}
for all $n\in\bbN$ because the homomorphism $\alpha$ maps $V_{-n}$ inside $V_{-n+1}$ by Lemma~\ref{lem:wil_lem2}, and $\alpha^{-1}(V_{-n})\cap V_{-n}=V_{-n-1}$.

Since the sequence $(t_{n})_{n\in\bbN_{0}}$ is non-negative, non-increasing and takes integer values it is eventually constant equal to some integer $t$. 
Since $\Gamma_{+}$ is a tree, we have $\smash{|Vv_{-n}|=\prod_{i=1}^{n-1}t_{i}}$. Given that $t_{i}= t$ for almost all $i\!\in\!\bbN_{0}$ there is a constant $l\in\bbQ$ such that $|Vv_{-n}|=lt^{n}$ for sufficiently large $n$. Then
\begin{displaymath}
k_{i}d_{i}^{n}\ge|Vv_{-n}|=lt^{n}\ge\frac{d_{i}^{n}}{k_{i}}
\end{displaymath}
for large enough $n\in\bbN$ and $i\in\{1,2\}$ by Lemma \ref{lem:orbit_bounds}. As a consequence, we have $d_{1}=t=d_{2}$ and thus $[U^{(1)}:U^{(1)}\cap\alpha^{-1}(U^{(1)})]=[U^{(2)}:U^{(2)}\cap\alpha^{-1}(U^{(2)})]$.
\end{proof}

The following theorem links the concept of being tidy to the scale function.

\begin{theorem}\label{thm:tidy_minimizing}
Let $G$ be a t.d.l.c. group, $\alpha\in\End(G)$ and $U\le G$ compact open. Then $U$ is tidy for $\alpha$ if and only if $U$ is minimizing for $\alpha$. In this case, $s(\alpha)=|\out_{\Gamma_{+}}(v_0)|$.
\end{theorem}

\begin{proof}
Suppose that $U$ is minimizing for $\alpha$. If $\{\alpha^{-k}(U)\mid k\in\bbN_{0}\}$ is finite, then $s(\alpha)=1$ by the definition of the scale function given in Section \ref{sec:Willis_theory} and Lemma~\ref{lem:finite_vertices_case}. Consequently, $\alpha(U)\le U$. Therefore, we have $U=U_{-}$ and $U_{--}\ge U_{-}=U$ is open and hence closed.

Assume now that $\{\alpha^{-k}(U)\mid k\in\bbN\}$ is infinite. First, we show that $U$ is tidy above for $\alpha$. Suppose otherwise. Then by Theorem \ref{thm:tidy_above_exists} and Lemma \ref{lem:moe_thm2.3_end} there is $n\in\bbN$ such that with $v_{-1}\in V(\Gamma)$ we have $|U_{n}v_{-1}|=|U_{+}v_{-1}|\lneq|Uv_{-1}|$ and such that $U_{-n}$ is tidy above for $\alpha$. Then
\begin{align*}
 [\alpha(U_{-n})&:\alpha(U_{-n})\cap U_{-n}]=[U_{-n}:U_{-n}\cap\alpha^{-1}(U_{-n})]=[U_{n}:U_{n}\cap\alpha^{-1}(U)] \\
 &=|U_{n}v_{-1}|\lneq|Uv_{-1}|=[U:U\cap\alpha^{-1}(U)]=[\alpha(U):\alpha(U)\cap U].
\end{align*}
where the equalities follow by applying the appropriate power of $\alpha$ to the respective quotient, using Lemma \ref{lem:wil_lem2}. This contradicts the assumption that $U$ is minimizing.

Now consider the graph $\Gamma_{++}$ associated to $U$ which has constant out-valency $d^{+}$, and constant in-valency $d^{-}$, excluding all $v\in V(\Gamma_{++})$ with $\psi(v)=0$, by Proposition \ref{prop:properties_of_gamma_++}. Since $U$ is tidy above, Remark \ref{rem:rel_gamma_++_gamma} implies that $d^{+}=|Uv_{-1}|=[U:U\cap\alpha^{-1}(U)]=[\alpha(U):\alpha(U)\cap U]$. Let $V$ denote the tidy subgroup constructed from the graph $\Gamma_{++}$ associated to $U$ by Theorem \ref{thm:tidy_existence}. Then the quotient $T$ of $\Gamma_{++}$ has out-valency
\begin{displaymath}
 d=[V:V\cap\alpha^{-1}(V)]=[\alpha(V):\alpha(V)\cap V].
\end{displaymath}
Furthermore, $d=d^{+}/d^{-}$ by Lemma \ref{lem:properties_of_T}. The fact that $U$ is minimizing now implies $d^{-}=1$. It follows that $\Gamma_{+} = \desc_{\Gamma_{++}}(v_0)$ is already a tree. Thus $U$ is tidy by Theorem~\ref{thm:tidy_tree}.

\vspace{0.2cm}
Conversely, assume that $U$ is tidy for $\alpha$. Let $V\le G$ be a compact open subgroup which is minimizing. Then $V$ is tidy by the above and Proposition \ref{prop:tidy_subgroups_index_equal} implies that
\begin{displaymath}
 s(\alpha)\!=\![\alpha(V):\alpha(V)\cap V]\!=\![V:V\cap\alpha^{-1}(V)]\!=\![U:U\cap\alpha^{-1}(U)]\!=\![\alpha(U):\alpha(U)\cap U].
\end{displaymath}
That is, $U$ is minimizing.
\end{proof}

\begin{corollary}\label{cor:scale_powers}
Let $G$ be a t.d.l.c. group and $\alpha\in\End(G)$. Then $s(\alpha^{n})=s(\alpha)^{n}$.
\end{corollary}

\begin{proof}
By Theorem \ref{thm:tidy_existence} there is a compact open subgroup $U\le G$ which is tidy for $\alpha$. Following Theorem \ref{thm:tidy_minimizing} the group $U$ is minimizing and therefore
\begin{displaymath}
 s(\alpha)=[\alpha(U):\alpha(U)\cap U]=[U:U\cap\alpha^{-1}(U)].
\end{displaymath}
Since $U$ is also tidy for $\alpha^{n}$ by Lemma \ref{lem:tidy_powers} we conclude, using the same lemma,~that
\begin{displaymath}
 s(\alpha^{n})=[\alpha^{n}(U):\alpha^{n}(U)\cap U]=[U:U\cap\alpha^{-n}(U)]=[U:U\cap\alpha^{-1}(U)]^{n}=s(\alpha)^{n}. \qedhere
\end{displaymath}
\end{proof}

M{\"o}ller's spectral radius formula \cite[Theorem 7.7]{Moe02} for the scale may be proven as in \cite[Proposition 18]{Wil15} but with reference to Theorem \ref{thm:tidy_existence} for the existence of tidy subgroups.

\begin{theorem}
Let $G$ be a t.d.l.c. group, $\alpha\in\End(G)$ and $U\le G$ compact open. Then $s(\alpha)=\lim_{n\to\infty}[\alpha^{n}(U):\alpha^{n}(U)\cap U]^{1/n}$. \qed
\end{theorem}

\section{The Tree-Representation Theorem}\label{sec:tree_rep_thm}

In this section, we prove an analogue of the following tree representation theorem for automorphisms due to Baumgartner and Willis \cite{BW04}, see also \cite{Hor15}.

\begin{theorem}[{\cite[Theorem 4.1]{BW04}}]\label{thm:tree_rep_bw}
Let $G$ be a t.d.l.c. group, $\alpha\!\in\!\Aut(G)$ of infinite order and $U\!\le\! G$ compact open as well as tidy for $\alpha$. Then there is a regular tree $T$ of degree $s(\alpha)+1$ and a homomorphism $\varphi:U_{++}\rtimes\langle \alpha\rangle\to \Aut(T)$ such that 
\begin{itemize}
  \item[(i)] $\varphi(U_{++}\rtimes\langle\alpha\rangle)$ fixes an end $\omega\in\partial T$ and is transitive on $\partial T\setminus \{\omega\}$,
  \item[(ii)] the stabilizer of each end in $\partial T\setminus \{\omega\}$ is conjugate to $(U_{+}\cap U_{-})\rtimes \langle \alpha \rangle$,
  \item[(iii)] $\ker(\varphi)$ is the largest compact normal subgroup $N\unlhd U_{++}$ with $\alpha(N)=N$,
  \item[(iv)] $\varphi(U_{++})$ is the set of elliptic elements in $\varphi(U_{++}\rtimes \langle\alpha\rangle)$.
\end{itemize}
\end{theorem}

To prove an analogous statement for endomorphisms, we let $\alpha\in\End(G)$ have infinite order and $U\le G$ compact open as well as tidy for $\alpha$. Let $S:=U_{++}\rtimes\langle\alpha\rangle$ be the topological semidirect product semigroup of the (semi)group $U_{++}\le G$ and the semigroup $\langle\alpha\rangle\le\End(G)$, where $\End(G)$ is equipped with the compact-open topology and $\langle\alpha\rangle$ acts continuously on $U_{++}$ by endomorphisms as $\alpha(U_{++})=U_{++}$, see \cite[Theorem 2.9, Theorem 2.10]{CHK83}. In particular:

\begin{enumerate}
  \item Elements of $S$ have the form $(u,\alpha^{k})$ for some $u\in U_{++}$ and $k\in\bbN_{0}$. We identify $(U_{++},\id)$ with $U_{++}$, and $(\id,\langle\alpha\rangle)$ with $\langle\alpha\rangle$.
  \item Composition in $S$ is given by $(u_0,\alpha^{k_0})(u_1,\alpha^{k_1}) = (u_0\alpha^{k_0}(u_1),\alpha^{k_0+k_1})$.
  \item The topology on $S$ is the product topology on the set $U_{++}\times\langle\alpha\rangle$.
  \item The subsemigroup of $S$ generated by $(\id, \alpha)$ is isomorphic to $(\bbN, +)$ because $\alpha\in\End(G)$ has infinite order.
\end{enumerate}

We split the construction of the desired tree into the cases $s(\alpha)=1$ and $s(\alpha)>1$. First, assume $s(\alpha)>1$. Recall that $v_{-i}:=\alpha^{-i}(U)\in\calP(G)$ for $i\ge 0$. We extend this definition to positive indices by setting $v_{i}:=\alpha^{i}(U)\in\calP(G)$ for all $i\in\bbZ$. The following lemma shows that these vertices are all distinct.

\begin{lemma}\label{lem:distinct_vertices}
Retain the above notation. In particular, assume $s(\alpha)>1$. Suppose $\alpha^{m}(U) = \alpha^{n}(U)$ for some $n,m\in\bbZ$. Then $m = n$.
\end{lemma}

\begin{proof}
For $m,n\le 0$, an equality $\alpha^{-m}(U)=\alpha^{-n}(U)$ with $m\neq n$ implies that the set $\{\alpha^{-k}(U)\mid k\in\bbN_{0}\}$ is finite and hence $s(\alpha)=1$ by Lemma \ref{lem:finite_vertices_case}.

Now, let $0\le m<n$. Since $U$ is tidy for $\alpha$, it is tidy for $\alpha^n$ and $\alpha^m$ by Lemma \ref{lem:tidy_powers}. Lemma \ref{lem:coset_calculations} shows that $s(\alpha^n) = [\alpha^n(U_+):U_+]$ and $s(\alpha^m) = [\alpha^m(U_+):U_+]$. The inclusions $U_+\le \alpha^m(U_+)\le \alpha^n(U_+)$, which can be seen from the description of $U_+$ as elements in $U$ which admit $\alpha$-regressive trajectories contained in $U$, give a factorization of $s(\alpha^n)$ as follows
\begin{align*} s(\alpha^n) &= [\alpha^n(U_+):U_+] = [\alpha^n(U_+):\alpha^m(U_+)][\alpha^m(U_+):U_+]\\
& = [\alpha^n(U_+):\alpha^m(U_+)]s(\alpha^m).
\end{align*} 
Applying Corollary \ref{cor:scale_powers} gives $s(\alpha)^n = [\alpha^n(U_+):\alpha^m(U_+)]s(\alpha)^m$. Since $m < n$ and $s(\alpha)>1$, we get $[\alpha^{n}(U_+):\alpha^{m}(U_+)]\!\neq\! 1$. So there is $u\!\in\! \alpha^{n}(U_+)\!\setminus \alpha^{m}(U_+)\!\subseteq\! \alpha^{n}(U)$. For the sake of a contradiction, suppose $u\in \alpha^{m}(U)$. Since $U$ is tidy above, there are $u_{\pm}\in U_{\pm}$ with $u =\alpha^{m}(u_+)\alpha^{m}(u_-)$. It follows that $\alpha^{m}(u_+)^{-1}u \in \alpha^{n}(U_+)\le U_{++}$ since $\alpha^{m}(U_{+})\le \alpha^{n}(U_{+})$. Also, we have $\alpha^{m}(u_-)\in \alpha^{m}(U_{-})\le U_{-}\le U_{--}$, and so applying Lemma \ref{lem:pseudo_nub}, 
\[\alpha^{m}(u_+)^{-1}u\in U_{++}\cap U_{--}\le U_+\cap U_-\le \alpha^{m}(U_+).\]
It follows that $u\in\alpha^{m}(U_+)$, a contradiction. Thus $u\not\in\alpha^{m}(U)$ and $\alpha^{n}(U)\!\neq\! \alpha^{m}(U)$.

Finally, suppose $m < 0 < n$ and $\alpha^{m}(U) = \alpha^{n}(U)$. Then $\alpha^{m}(U)$ is a compact open subgroup which is stabilized by $\alpha^{n-m}$. This shows $s(\alpha^{n - m}) = 1$ which implies $s(\alpha) = 1$ by Corollary \ref{cor:scale_powers}. This contradicts the assumption $s(\alpha) > 1$.
\end{proof}

We define a directed graph $\overline{\Gamma}_{++}$ by setting 
\[V(\overline{\Gamma}_{++}) = \{uv_{i}\mid i\in\bbZ,u\in U_{++}\} \ \text{ and }\ E(\overline{\Gamma}_{++}) = \{(uv_{i},uv_{i - 1}\mid i\in\bbZ,u\in U_{++}\}.\]

Note that $\Gamma_{++}$ is a subgraph of $\overline{\Gamma}_{++}$ and that $U_{++}$ acts on $\overline{\Gamma}_{++}$ by automorphisms. We will show that the map $\rho$, defined in the paragraph preceding Proposition \ref{prop:properties_of_induced_map}, extends to an automorphism of $\overline{\Gamma}_{++}$. To do so, consider the following subgroups associated to $\alpha$:
\begin{displaymath}
  \parb^{-}(\alpha):=\{x\in G\mid\hbox{there exists a bounded }\alpha\hbox{-regressive trajectory for }x\},
\end{displaymath}
\begin{displaymath}
  \bik(\alpha):= \overline{\{x\in \parb^{-}(\alpha)\mid\alpha^{n}(x) = e \hbox{ for some }n\in\bbN\}}.
\end{displaymath}
It follows from \cite[Proposition 20]{Wil15}, \cite[Definition 12]{Wil15} and Theorem \ref{thm:tidy_minimizing} that $\alpha(\bik(\alpha)) = \bik(\alpha)\le U$. The same proposition implies that for $u_1,u_2\in U_{++}\le \parb^{-}(\alpha)$ with $\alpha(u_1) = \alpha(u_2)$ we have $u_1^{-1}u_2\in\bik(\alpha)\le U$.

Now define $\rho: \overline{\Gamma}_{++}\to \overline{\Gamma}_{++}$ as follows: Given $uv_{i}\in V(\overline{\Gamma}_{++})$, choose $u'\in U_{++}$ such that $\alpha(u') = u$ and set $\rho(uv_{i}) = u'v_{i-1}$.

\begin{proposition}
Retain the above notation. The map $\rho$ is an automorphism of the graph $\overline{\Gamma}_{++}$.
\end{proposition}

\begin{proof}
We first show that $\rho$ is well-defined: By Lemma \ref{lem:distinct_vertices}, it suffices to suppose $u_0,u_1,u_0',u_1'\in U_{++}$ and $i\in\bbZ$ are such that $u_0v_{i} \!=\! u_1v_{i}$, $\alpha(u_0') \!=\! u_0$ and $\alpha(u_1') \!=\! u_1$. Then $u_0^{-1}u_1\in \alpha^{i}(U)$ and $(u_0')^{-1}u_1'\in \alpha^{-1}(\alpha^{i}(U))\cap U_{++}$.

For $i\ge 1$, we have $u_{0}^{-1}u_{1}\in\alpha^{i}(U)=\alpha(\alpha^{i-1}(U))$, hence there is $u_3\in \alpha^{i-1}(U)$ with $\alpha(u_3) = u_0^{-1}u_1$ and we get $((u_0')^{-1}u_1')^{-1}u_3\in \bik(\alpha)\le \alpha^{i-1}(U)$ as $\bik(\alpha)\le U$ and $\alpha(\bik(\alpha))=\bik(\alpha)$. Thus $(u_0')^{-1}u_1'\in \alpha^{i-1}(U)$. This shows $u_0'v_{i - 1} = u_1'v_{i - 1}$.

For $i\le 0$, we argue as in the proof of Proposition \ref{prop:properties_of_induced_map}: We have $\alpha((u_0')^{-1}u_{1}') = u_{0}^{-1}u_1\in \alpha^{i}(U)$ and so $(u_0')^{-1}u_1'\in \alpha^{i-1}(U)$. This implies $w_0'v_{i-1} = u_{1}'v_{i-1}$. Overall, $\rho$ is well-defined.

To see that $\rho$ is a bijection on $V(\overline{\Gamma}_{++})$ note $\rho(\alpha(u)v_{i+1}) = uv_i$ and that $\rho^{-1}$ defined by $uv_{i}\mapsto \alpha(u)v_{i+1}$ is well-defined by the following argument: If $uv_{i} =u'v_{i}$, then $u^{-1}u'\in \alpha^{i}(U)$ and $\alpha(u)^{-1}\alpha(u')\in \alpha^{i+1}(U)$. Thus $\alpha(u)v_{i+1} = \alpha(u')v_{i+1}$.
\end{proof}

Note that $\overline{\Gamma}_{++}$ contains $\Gamma_{++}$ as a subgraph and $\Gamma_{++}$ is a forest of rooted regular trees by Remark \ref{rem:gamma_++_tidy} and Lemma \ref{lem:properties_of_T}. For $v\in V(\overline{\Gamma}_{++})$, there is $n\in\bbN_{0}$ such that $\rho^{n}(v)\in V(\Gamma_{++})$. This shows that the in-valency of $v$ is $1$. 

We find that $\overline{\Gamma}_{++}$ is a regular tree with constant out-valency because $\langle U_{++},\rho\rangle$ acts transitively on $V(\overline{\Gamma}_{++})$. Also, $|\out_{\overline{\Gamma}_{++}}(v_{0})|=[U\cap U_{++}:U\cap U_{++}\cap\alpha^{-1}(U)]$ is equal to $[U:U\cap\alpha^{-1}(U)]$ by Remark \ref{rem:rel_gamma_++_gamma} which in turn is equal to $s(\alpha)$ by Theorem \ref{thm:tidy_minimizing} and the equality $[U:U\cap\alpha^{-1}(U)]=[\alpha(U):\alpha(U)\cap U]$.

Since $\rho$ is a translation in $\Aut(\overline{\Gamma}_{++})$ we see that the subsemigroup generated by $\rho^{-1}$ is isomorphic to $(\bbN,+)$.

Define $\varphi: U_{++}\sqcup \langle \alpha \rangle \to \Aut(\overline{\Gamma}_{++})$ by $\varphi(u)(u'v_i)\! =\! uu'v_i$ for all $u,u'\in U_{++}$ and $\varphi(\alpha^{k}) = \rho^{-k}$ for all $k\in\bbN_{0}$.

\begin{lemma}\label{lem:tree_rep_cont}
Retain the above notation. The map $\varphi$ extends to a continuous semigroup homomorphism $\varphi:S\to\Aut(\overline{\Gamma}_{++})$. 
\end{lemma}

\begin{proof}
Note that $\varphi$ extends separately both to a semigroup homomorphism of $U_{++}$, and the semigroup generated by $\alpha$. To show that it extends to a semigroup homomorphism of $S$ it suffices to show that $\varphi(\alpha)\varphi(u) = \varphi(\alpha(u))\varphi(\alpha)$. Then $\varphi(u,\alpha^{n}):=\varphi(u)\varphi(\alpha^{n})$ is well-defined for all $u\in U_{++}$ and $n\in\bbN_0$. Given a vertex $u'v_{i}\in V(\overline{\Gamma}_{++})$, we obtain as required:
\[\varphi(\alpha)\varphi(u)u'v_{i} = \rho^{-1}(uu'v_{i}) = \alpha(uu')v_{i+1} = \alpha(u)\rho^{-1}(u'v_{i}) = \varphi(\alpha(u))\varphi(\alpha)u'v_{i}.\]
To see that $\varphi$ is continuous it suffices to show that $\{x\in S\mid \varphi(x)w = w'\}$ is open in $S$ for all $w,w'\in V(\overline{\Gamma}_{++})$. We first show that the subgroup of $U_{++}$ stabilising $w'$ is open in $U_{++}$. Write $w' = u\alpha^{i}(U)$ where $u\in U_{++}$ and $i\in\bbZ$. The stabilizer is $u\alpha^i(U)u^{-1}\cap U_{++}$ which contains $u(\alpha^i(U)\cap U_{++})u^{-1}$. It suffices to show that $\alpha^i(U)\cap U_{++}$ is open in $U_{++}$. If $i\le 0$, then continuity of $\alpha$ implies $\alpha^i(U)$ is open in $G$. It follows that $\alpha^i(U)\cap U_{++}$ is open in $U_{++}$. Otherwise $i > 0$ and $U_+\le \alpha^i(U_+)\le \alpha^i(U)$. So we have $U_+\le \alpha^i(U)\cap U_{++}$. But $U$ is assumed to be tidy for $\alpha$ and so $U_{+} = U\cap U_{++}$ by Lemma \ref{lem:tidy_below}. This is an open set in $U_{++}$.  

We have shown that the stabilizer $V$ of $w'$ in $U_{++}$ is an open subgroup of $U_{++}$, so $x\in S$ with $x(w) = w'$ is contained in the open subset $(V,\id)x\subseteq S$ and $\varphi((V,\id)x)w = w'$.

\end{proof}
The following Lemma extends Lemma \ref{lem:tidy_below}\ref{item:lem:tidy_below:U_+}.
\begin{lemma}\label{lem:precise_preimage}
Retain the above notation. Let $x_0\in \alpha^{i}(U)\cap U_{++}$ ($i\in\bbZ$). Then there is $x_1\in \alpha^{i-1}(U)\cap U_{++}$ with $\alpha(x_1) = x_0$.
Hence, $\alpha^{i}(U)\cap U_{++} = \alpha^{i}(U_+)\cap U_{++}$.
\end{lemma}

\begin{proof}
First, suppose $i\ge 0$ and let $x'\in U$ with $\alpha^{i}(x')=x_{0}$. Since $U$ is tidy there are $x_{+}'\in U_{+}$ and $x_{-}'\in U_{-}$ such that $x'=x_{+}'x_{-}'$. Then $x_{0}=\alpha^{i}(x_{+}')\alpha^{i}(x_{-}')$ where $\alpha^{i}(x_{+}')\in U_{++}$ and $\alpha^{i+1}(x_{-}')\in U_{-}$. Given that $x_{0}\in U_{++}$ we conclude that $\alpha^{i}(x_{-}')\in U_{-}\cap U_{++}\le U_{-}\cap U_{+}$ by Lemma \ref{lem:tidy_below}. Hence we may pick $x''\in U_{+}$ with $\alpha(x_{+}'')=\alpha^{i}(x_{-}')\in U_{+}\cap U_{-}$. Then $\alpha(\alpha^{i-1}(x_{+}')x'')=\alpha^{i}(x_{+}')\alpha(x'')=x_{0}$ where $\alpha^{i-1}(x_{+}')x''\in U_{++}\cap\alpha^{i}(U)$ as $x''\in U_{+}$. So we may set $x_{1}:=\alpha^{i-1}(x_{+}')x''$ which gives the first claim when $i \ge 0$.
By repeating this process, we find $x_i\in U\cap U_{++}$ such that $\alpha^i(x_i) = x_0$. Lemma \ref{lem:tidy_below} gives $U\cap U_{++} = U_{+}$ and so $x_0\in \alpha^i(U_+)$. We conclude that $\alpha^i(U)\cap U_{++}\le \alpha^{i}(U_+)\cap U_{++}$.

If $i<0$, pick $x_{1}\in U_{++}$ with $\alpha(x_{1})=x_{0}$. Then $\alpha(x_{1})\in\alpha^{i}(U)$ and therefore $x_{1}\in\alpha^{-1}(\alpha^{i}(U))=\alpha^{i-1}(U)$ as $i<0$. Furthermore, $\alpha^{-i}(x_0)\in U\cap U_{++} = U_{+}$ and thus $x_0\in \alpha^{i}(U_+)$. This shows $\alpha^i(U)\cap U_{++}\le \alpha^{i}(U_+)\cap U_{++}$. 

Finally, note that $\alpha^{i}(U_+)\cap U_{++}\le \alpha^i(U)\cap U_{++}$ for all $i\in\bbZ$ as $U_{+}\le U$.
\end{proof}

We are now in position to prove an analogue of Theorem \ref{thm:tree_rep_bw} for endomorphisms.

\begin{theorem}\label{thm:tree_rep_thm}
Let $G$ be a t.d.l.c. group, $\alpha\!\in\!\End(G)$ of infinite order, $U\le G$ compact open as well as tidy for $\alpha$, and $S:=U_{++}\rtimes\langle\alpha\rangle$. Then there is a tree $T$ and a continuous semigroup homomorphism $\varphi:S\to\Aut(T)$ such that
\begin{itemize}
  \item[(i)] $T$ has constant valency $s(\alpha)+1$,
  \item[(ii)] $\varphi(S)$ fixes an end $\omega\in\partial T$ and is transitive on $\partial T\setminus \{\omega\}$,	
  \item[(iii)] $\ker(\varphi)$ is the largest compact normal subgroup $N\unlhd U_{++}$ with $\alpha(N)=N$,
  \item[(iv)] $\varphi(U_{++})$ is the set of elliptic elements of $\varphi(S)$.
\end{itemize}
\end{theorem}

\begin{proof}
First, assume $s(\alpha)>1$. Let $T$ be the undirected graph underlying $\overline{\Gamma}_{++}$, i.e. the graph with vertex set $V(\overline{\Gamma}_{++})$ and edge-relation the symmetric closure of $E(\overline{\Gamma}_{++})\subseteq V(\overline{\Gamma}_{++})\times V(\overline{\Gamma}_{++})$. The continuous semigroup homomorphism $\varphi$ from $S$ to $\Aut(\overline{\Gamma}_{++})$ defined above induces a continuous semigroup homomorphism $S\to \Aut(T)$ for which we use the same letter.

Part (i) is now immediate from the fact that every vertex in $\overline{\Gamma}_{++}$ has out-valency $s(\alpha)$ and in-valency $1$.

For part (ii), let $\omega\in\partial T$ be the end associated to the sequence $(v_i)_{i\in\bbN_{0}}$. Then $\varphi(\alpha)=\rho^{-1}$ fixes $\omega$ by definition. If $u\in U_{++}$, then there exists an $\alpha$-regressive trajectory for $u$ eventually contained in $U$. That is $u\in\alpha^{n}(U)$ for all sufficiently large $n\in\bbN$ whence $uv_{n} = v_{n}$ for sufficiently large $n$. This shows that $u\omega=\omega$. Overall, we conclude $\varphi(S)\omega=\omega$. 

Now consider the end $-\omega\in\partial T$ associated to the sequence $(v_{-i})_{i\in\bbN_{0}}$ and let $\omega'\in\partial T$ be an end distinct from $\omega$. Because $T$ is a regular tree with constant in-valency $1$ and $V(\overline{\Gamma}_{++})=\{uv_{i}\mid u\in U_{++},\ i\in\bbZ\}$, the end $\omega'$ is represented by $(u_{k-i}v_{k-i})_{i\in\bbN_{0}}$ for some $k\in\bbZ$ and some sequence $(u_{k-i})_{i\in\bbN_{0}}$ in $U_{++}$. Then $u_{k}^{-1}\rho^{k}\omega'$ is an end $\omega''\in\partial T$ represented by a sequence originating from $v_{0}$ and it suffices to show that there is an element $u\in U_{++}$ which maps the sequence of $-\omega$ to that of $\omega''$. Since $U\cap U_{++}=U_{+}$ by Lemma \ref{lem:tidy_below} as $U$ is tidy, the set $U\cap U_{++}$ is a descending intersection of compact subgroups and so is compact. We finish by applying Lemma \ref{lem:basic_arc_trans} and by picking a convergent subsequence inside $U\cap U_{++}$.

As to (iii), let $N$ be any compact normal subgroup of $U_{++}$ with $\alpha(N)=N$. We show that $N\!\le\! \ker \varphi$ before showing that $\alpha(\ker(\varphi)) = \ker(\varphi)$. By \cite[Theorem 5.2]{FTN91}, $\varphi(N)\le\Aut(\overline{\Gamma}_{++})_{v}$ for some $v\in V(\overline{\Gamma}_{++})$ because $\varphi(N)$ is compact by Lemma~\ref{lem:tree_rep_cont} and acts without inversions. Since $N$ is normal in $U_{++}$, we obtain
\begin{displaymath}
\varphi(N)=\varphi(u)\varphi(N)\varphi(u)^{-1}\le\varphi(N)\cap\Aut(\overline{\Gamma}_{++})_{\varphi(u)v}\le\Aut(\overline{\Gamma}_{++})_{v,\varphi(u)v}
\end{displaymath}
for all $u\in U_{++}$. Similarly, given that $\alpha(N)=N$ we have
\begin{align*}
\varphi(N)=\varphi(\alpha(N))&\varphi(\alpha)\varphi(\alpha)^{-1}=\varphi(\alpha(N)\circ\alpha)\varphi(\alpha)^{-1} \\
&=\varphi(\alpha\circ N)\varphi(\alpha)^{-1}=\rho^{-1}\varphi(N)\rho\le\Aut(\overline{\Gamma}_{++})_{v,\rho^{-1}(v)}.
\end{align*}
as well as 
\begin{align*}
\varphi(N)=\varphi(\alpha)^{-1}&\varphi(\alpha)\varphi(N)=\varphi(\alpha)^{-1}\varphi(\alpha\circ N) \\
&=\varphi(\alpha)^{-1}\varphi(\alpha(N))\varphi(\alpha)=\rho\varphi(N)\rho^{-1}\le\Aut(\overline{\Gamma}_{++})_{v,\rho(v)}.
\end{align*}
As a consequence, $\varphi(N)$ fixes every vertex in the orbit of $v$ under the action of the group generated by $\varphi(S)$. This group acts vertex-transitively as it contains $\varphi(U_{++})$ and both $\rho$ and $\rho^{-1}$. This shows that $\varphi(N)$ fixes $T$, i.e. $N\le\ker(\varphi)$.

The kernel of $\varphi$ consists of those elements $s\in S$ such that $\varphi(s)$ fixes every vertex of $T$. Note that $\ker(\varphi)\le U_{++}\le S$ by Lemma \ref{lem:distinct_vertices}, so
\begin{displaymath}
  \ker(\varphi)=U_{++}\cap\bigcap_{i\in\bbZ}\bigcap_{u\in U_{++}}u\alpha^{i}(U)u^{-1}.
\end{displaymath}
In particular, $\ker(\varphi)$ is compact and $\alpha(\ker(\varphi))\le\ker(\varphi)$ as $\alpha(U_{++})=U_{++}$. To see that $\alpha(\ker(\varphi)) = \ker(\varphi)$, suppose $u_0\in\ker(\varphi)$ and choose $u_1\in U_{++}$ such that $\alpha(u_1) = u_0$. Such a $u_1$ exists as $u_0\in U_{++}$. We will show $u_1\in \ker(\varphi)$. Let $w\in U_{++}$ and $i\in\bbZ$. Then $u_0\in \alpha(w)\alpha^{i+1}(U)\alpha(w)^{-1}$ as $u_0\in\ker(\varphi)$. Choose $x_{0}\in\alpha^{i+1}(U)$ such that $u_0 = \alpha(w)x_{0}\alpha(w^{-1})$. Then $x_{0}\in U_{++}\cap\alpha^{i+1}(U)$ as $u_{0}$, $\alpha(w)$, $\alpha(w^{-1})\in U_{++}$. Lemma \ref{lem:precise_preimage} gives $x_{1}\in\alpha^{i}(U)\cap U_{++}$ such that $\alpha(x_{1})=x_{0}$. We obtain $u_1\bik(\alpha) = wx_{1}w^{-1}\bik(\alpha)$. However, $\alpha(\bik(\alpha)) = \bik(\alpha)$ and $\bik(\alpha)$ is compact and normal in $U_{++}$. We have already shown that such a subgroup is contained in $\ker(\varphi)\le w\alpha^{i}(U)w^{-1}$, hence $u_1\in w\alpha^{i}(U)w^{-1}$. As our choice of $w\in U_{++}$ and $i\in\bbZ$ was arbitrary, $u_1\in\ker(\varphi)$.

For part (iv), write $s=(u,\alpha^{k})$ $(u\in U_{++},\ k\in\bbN)$ for elements of $S$. Given that $\varphi(\alpha)=\rho^{-1}$, we necessarily have $k=0$ in order for $\varphi(s)$ to fix a vertex, so $s\in U_{++}$. Conversely, every element $u\in U_{++}$ is contained in $\alpha^{n}(U)$ for all sufficiently large $n\in\bbN$, so $\varphi(u)$ fixes $v_{n}$ for the same values of $n$.

Now, assume $s(\alpha) = 1$. Then $\alpha(U_+) = U_+$ by Lemma \ref{lem:coset_calculations}. This shows that $U_{++} = U_{+}$ is a compact subgroup with $\alpha(U_{++}) = U_{++}$. Let $T$ be the (undirected) tree with vertex set $\bbZ$ and $i,j\in V(T)$ connected by an edge whenever $|i-j|=1$. Define $\varphi:S\to\Aut(T)$ by setting $\varphi(\alpha)$ to be the translation of length $1$ in the direction of $\omega:=(i)_{i\in\bbN_{0}}\in\partial T$, and $\varphi(u)$ to be the identity automorphism of $T$ for all $u\in U_{++}$. Then $\varphi$ satisfies all the conclusions of Theorem \ref{thm:tree_rep_thm}. 
\end{proof}

\begin{remark}
The action in Theorem \ref{thm:tree_rep_thm} relates to Theorem \ref{thm:tree_rep_bw} in the following manner: Results from \cite[Section 9]{Wil15} show that if $U$ is tidy for $\alpha$, then $\bik(\alpha)\unlhd U_{++}$ and the endomorphism $\overline{\alpha}$ of $U_{++}/\bik(\alpha)$ induced by $\alpha|_{U_{++}}$ is an automorphism. Let $q:U_{++}\to U_{++}/\bik(\alpha)$ be the quotient map. Then $q(U_+)$ is tidy for $\overline{\alpha}$, $(q(U_{+}))_{++} = q(U_{++})$ and $s(\overline{\alpha}) = s(\alpha)$. Extend $q$ to a semigroup homomorphism from $S$ to $q(U_{++})\rtimes\langle \overline{\alpha}\rangle$ by setting $q(\alpha) = \overline{\alpha}$. Also, let $\varphi: S\to\Aut(T)$ be as in Theorem \ref{thm:tree_rep_thm} and $\varphi': q(U_{++})\rtimes\langle \overline{\alpha}\rangle \to T'$ as in Theorem \ref{thm:tree_rep_bw}. Then there exists a graph isomorphism $\psi:T'\to T$ such that the diagram
\begin{displaymath}
 \xymatrix{
  S \ar[r]^-{\varphi} \ar[d]_{q} & \Aut(T) \\
  q(U_{++})\rtimes\langle\overline{\alpha}\rangle \ar[r]_-{\varphi'} & \Aut(T') \ar[u]_{\widetilde{\psi}},
 }
\end{displaymath}
where $\smash{\widetilde{\psi}}$ is conjugation by $\psi$, commutes.
\end{remark}

\section{New Endomorphisms From Old}
\label{sec:new_from_old}
We conclude with a construction that produces new endomorphisms of totally disconnected, locally compact groups from old, inspired by \cite[Example 5]{Wil15}.

Let $G_{1}$ and $G_{2}$ be totally disconnected compact groups. Assume that there are isomorphisms $\varphi_{i}:G_{i}\to H_{i}$ ($i\in\{1,2\}$) of $G_{i}$ onto compact open subgroups $H_{i}\le G_{i}$. Consider the HNN-extension $G$ of $G_{1}\times G_{2}$ which makes the isomorphic subgroups $H_{1}\times G_{2}\cong G_{1}\times G_{2}\cong G_{1}\times H_{2}$ conjugate:
\begin{displaymath}
 G:=\langle G_{1}\!\times\! G_{2},t\mid\{t^{-1}(h_{1},g_{2})t=(\varphi_{1}^{-1}(h_{1}),\varphi_{2}(g_{2}))\mid (h_{1},g_{2})\in H_{1}\!\times\! G_{2}\}\rangle.
\end{displaymath}
Set $U:=G_{1}\times G_{2}\le G$. Given that $G$ commensurates $U$, it admits a unique group topology which makes the inclusion of $U$ into $G$ continuous and open, see \cite[Chapter III, \S 1.2, Proposition 1]{Bou98}. Then $G$ is a non-compact t.d.l.c. group which contains $U\!:=\!G_{1}\!\times\! G_{2}$ as a compact open subgroup. Define $\beta\in\End(G)$ by setting $\beta(t)=t$ and $\beta(g_{1},g_{2})=(\varphi_{1}(g_{1}),g_{2})$ for all $(g_{1},g_{2})\in G_{1}\times G_{2}$. Then
\begin{displaymath}
  \beta(t^{-1}(h_{1},g_{2})t)=t^{-1}(\varphi_{1}(h_{1}),g_{2})t=(h_{1},g_{2})=\beta(\varphi_{1}^{-1}(h_{1}),g_{2}).
\end{displaymath}
for all $(h_{1},g_{2})\in H_{1}\times G_{2}$ and hence $\beta$ indeed extends to $G$. Note that $\beta$ is continuous: Let $V\le G$ be open. Then so is $V\cap(H_{1}\cap G_{2})$ and $\beta^{-1}(V)\supseteq\beta^{-1}(V\cap(H_{1}\cap G_{2}))\cap U$ which is open in $U$ and therefore in $G$ since $\varphi_{1}$ is continuous. Observe that $s(\beta)=1$ as $\beta(U)\le U$. Let $\alpha:=c_{t}\circ\beta\in\End(G)$ where $c_{t}:G\to G,\ g\mapsto tgt^{-1}$ is conjugation by $t$. For $(g_{1},h_{2})\in G_{1}\!\times\! H_{2}$ we have
\begin{equation}
 \alpha(g_{1},h_{2})=t\beta(g_{1},h_{2})t^{-1}=t(\varphi_{1}(g_{1}),h_{2})t^{-1}=(\varphi_{1}^{2}(g_{1}),\varphi_{2}^{-1}(h_{2}))
 \tag{E}
 \label{eq:def_alpha}
\end{equation}

\begin{lemma}\label{lem:hnn_tidy}
Retain the above notation. Then $U$ is tidy for $\alpha$ and $s(\alpha)\! =\! [G_2:H_2]$.
\end{lemma}

\begin{proof}
We proceed via Lemma \ref{lem:tidy_powers}. First, we show that $\alpha^{-n}(U)\!\cap\! U\! =\! G_1\!\times\!\varphi_2^{n}(G_2)$. The inclusion $ G_1\times \varphi_2^{n}(G_2)\le \alpha^{-n}(U)\cap U$ follows from equation \eqref{eq:def_alpha}. Suppose $g\not\in G_1\times \varphi_2^{n}(G_2)$. We will show $g\not\in \alpha^{-n}(U)\cap U$. If $g\not\in U$, then we are done and so we may write $g = (g_1,g_2)\in G_1\times (G_2\setminus \varphi_2^{n}(G_2))$. By equation \eqref{eq:def_alpha}, there exists $0\le m < n$ such that $\alpha^{m}(g_1,g_2)\in G_1\times (G_2\setminus H_2)$. We therefore show that $\alpha^{l}(g'_{1},g'_{2})\not\in U$ for all $l\in\bbN$ whenever $(g'_1,g'_{2})\in G_1\times( G_{2}\backslash H_{2})$. Indeed, $\alpha^{l}(g_{1},g_{2})=t^{l}(\varphi_{1}^{l}(g_{1}),g_{2})t^{-l}$ is not contained in $U$: If $t^{l}(\varphi_{1}^{l}(g_{1}),g_{2})t^{-l}=(g_{1}',g_{2}')\in U$, then
\begin{displaymath}
t\cdots t(\varphi_{1}^{l}(g_{1}),g_{2})t^{-1}\cdots t^{-1}(g_{1}'^{-1},g_{2}'^{-1})=1, 
\end{displaymath}
contradicting Britton's Lemma on words in HNN-extensions, see \cite[Lemma 4]{Bri63} or \cite[Theorem 2.1]{LS15}.

We have shown that $\alpha^{-n}(U)\cap U = G_1\times \varphi_2^{n}(G_2)$. Since $\varphi_2^{n}(G_2)$ is a nested series of subgroups for $n\in\bbN$, we have
\begin{align*}
[U:U\cap \alpha^{-n}(U)] &= [G_1\times G_2:G_1\times \varphi_2^{n}(G_2)] = [G_2:\varphi_2^{n}(G_2)]\\ 
&= \prod_{i = 0}^{n - 1}[\varphi_2^{i}(G_2):\varphi_{2}^{i+1}(G_2)] = [G_2:H_2]^{n}.
\end{align*}
Lemma \ref{lem:tidy_powers} shows that $U$ is tidy. By Theorem \ref{thm:tidy_minimizing}, we have
\begin{displaymath}
  s(\alpha)=[U:U\cap\alpha^{-1}(U)]=[G_{1}\!\times\! G_{2}:G_{1}\!\times\! H_{2}]=[G_{2}:H_{2}]. \qedhere
\end{displaymath}
\end{proof}

\bibliographystyle{plain}
\bibliography{willis_theory_graphs}

\end{document}